\documentclass[oneside,reqno]{amsart}
\usepackage{amsfonts,amsmath,latexsym,verbatim,amscd,mathrsfs,color,array}
\usepackage[colorlinks=true]{hyperref}

\usepackage{amsmath,amssymb,amsthm,amsfonts,graphicx,color}
\usepackage{amssymb}
\usepackage{pdfsync}
\usepackage{epstopdf}

\newcommand{\ass}{\quad\mbox{as}\quad}

\newcommand{\inn}{{\quad\hbox{in } }}

\newcommand{\ttt}{\tilde }

\newcommand{\TT}{{\mathcal T}  }

\newcommand{\nn}{ {\nabla}  }

\newcommand{\pp}{ {\partial} }

\newcommand{\vp}{\varphi}

\newcommand{\R} {\mathbb R}

\newcommand{\cuad}{{\sqcap\kern-.68em\sqcup}}

\newcommand{\DD}{{\mathcal D}}

\newcommand{\foral}{\quad\mbox{for all}\quad}
\newcommand{\ve}{\varepsilon}

\newcommand{\be}{\begin{equation}}
\newcommand{\ee}{\end{equation}}

\newcommand{\la}{\lambda}

\newcommand{\equ}[1]{(\ref{#1})}

\newtheorem{lemma}{Lemma}[section]
\newtheorem{prop}{Proposition}[section]
\newtheorem{theorem}{Theorem}
\newtheorem{corollary}{Corollary}
\newtheorem{remark}{Remark}[section]
\newcommand{\bremark}{\begin{remark} \em}
\newcommand{\eremark}{\end{remark} }

\numberwithin{equation}{section}

\begin{document}

\title[Bubble towers in the critical heat equation]{Existence and stability of infinite time bubble towers in the energy critical heat equation}

\author[M. del Pino]{Manuel del Pino}
\address{\noindent   Department of Mathematical Sciences University of Bath,
Bath BA2 7AY, United Kingdom \\
and  Departamento de
Ingenier\'{\i}a  Matem\'atica-CMM   Universidad de Chile,
Santiago 837-0456, Chile}
\email{m.delpino@bath.ac.uk}

\author[M. Musso]{Monica Musso}
\address{\noindent   Department of Mathematical Sciences University of Bath,
Bath BA2 7AY, United Kingdom}
\email{m.musso@bath.ac.uk}

\author[J. Wei]{Juncheng Wei}
\address{\noindent  Department of Mathematics University of British Columbia, Vancouver, BC V6T 1Z2, Canada
}  \email{jcwei@math.ubc.ca}

\begin{abstract} We consider
the energy critical heat equation in $\R^n$ for $n\ge 7$
$$\left\{
\begin{aligned}
u_t   & = \Delta u+ |u|^{\frac 4{n-2}}u   \inn\ \R^n \times (0, \infty),  \\
u(\cdot,0) &   = u_0  \inn \R^n,
\end{aligned}\right.
$$
which corresponds to the $L^2$-gradient flow of the Sobolev-critical energy
$$
J(u) = \int_{\R^n} e[u] , \quad  e[u] :=  \frac 12 |\nn u|^2  -   \frac {n-2}{2n} |u|^{\frac {2n}{n-2} }.
$$
Given any $k\ge 2$ we find an initial condition $u_0$ that leads
to sign-changing  solutions with {\em multiple blow-up at a single point} (tower of bubbles) as $t\to +\infty$.
 It has the form
of a superposition with alternate signs of singularly scaled {\em Aubin-Talenti solitons},
$$
u(x,t) =  \sum_{j=1}^k  (-1)^{j-1} {\mu_j^{-\frac {n-2}2}} U \left( \frac {x}{\mu_j} \right)\,  +\,  o(1) \ass t\to +\infty
$$
where $U(y)$ is the standard soliton
 $
 U(y)  =   
 \alpha_n\left ( \frac 1{1+|y|^2}\right)^{\frac{n-2}2}$
and
 $$\mu_j(t) = \beta_j t^{- \alpha_j}, \quad  \alpha_j =    \frac 12   \Big ( \, \left( \frac{n-2}{n-6}\right)^{j-1} -1 \Big). $$
Letting $\delta_0$ the Dirac mass, we have energy  concentration of the form
$$
e[ u(\cdot, t)]- e[U]  \rightharpoonup     (k-1) S_n\,\delta_{0} \ass t\to +\infty
$$
where $S_n=J(U)$.
The initial condition can be chosen radial and compactly supported. We establish the codimension $k+ n (k-1)$ stability of
this phenomenon for  perturbations of the initial condition that have
 space decay $u_0(x) =O( |x|^{-\alpha})$, $\alpha > \frac {n-2}2$, which yields finite energy of the solution.


\end{abstract}

\maketitle

\section{Introduction}
This paper deals with the analysis of solutions that exhibit {\em infinite time blow-up} in the energy critical heat equation
\be\left\{
\begin{aligned}
u_t   & = \Delta u+ |u|^{p-1}u   \inn\ \R^n \times (0, \infty),  \\
u(\cdot,0) &  =u_0 \inn \R^n
\end{aligned}\right.
\label{F}\ee
where $n\ge 3 $ and $p$ is the critical Sobolev exponent
$p= \frac{n+2}{n-2}$.
We are interested in solutions $u(x,t)$ globally defined in time such that
$$ \lim_{t\to +\infty } \| u(\cdot ,t)\|_{L^\infty (\R^n)} = + \infty.$$
The {\em energy functional} associated to \equ{F} is the functional
$$
J(u) = \int_{\R^n} e[u] , \quad e[u]:=  \frac 12|\nn u|^2  -   \frac 1{p+1}|u|^{p+1}
$$
which represents a Lyapunov functional for \equ{F} in the sense that $t\mapsto J(u(\cdot ,t)) $ is decreasing. In fact for a solution globally defined in time we must have $J(u(\cdot ,t))\ge 0$ for all $t$ and hence the value $\lim_{t\to+\infty }  J(u(\cdot, t) )$ exists and it is
nonnegative.

\medskip
 The behavior at infinity for
finite energy solutions is of course connected to steady states, namely solutions of the Yamabe equation
 \be\label{Y}
 \Delta u  +  |u|^{\frac 4{n-2}}u = 0 \inn \R^n.
 \ee
 $u(\cdot,t)$ as $t\to +\infty$ is along sequences $=t_n \to+\infty$, of Palais-Smale type for the energy $J$. 
 An application of the classical Struwe's profile decomposition \cite{struwe} (in a form given in \cite{gerrard})
tells us that passing to a subsequence,  there are finite energy solutions $U_1$, \ldots, $U_k$ of \equ{Y}, positive scalars $\mu_j(t) $ and points
$\xi_j(t)$
such that  for $i\ne j$,
$$
 \Big |\log \frac{\mu_1}{\mu_j}(t)\Big|   + \frac {\xi_i-\xi_j} {\mu_i}(t) \to +\infty \ass t=t_n\to +\infty
$$
and
\be \label{PP} u(x,t)  =   \sum_{j=1}^k  \frac 1{\mu_j(t)^{\frac{n-2}2}} U_j \left (\frac{x-\xi_j(t) } {\mu_j(t)}  \right)   + o(1)\ass t=t_n\to +\infty   \ee

This information is vague, since no information can be directly drawn from the centers and the scaling parameters. Even worse, the steady states could in principle depend on the particular sequence chosen. It is therefore a natural question to understand in which precise ways a profile decomposition like
\equ{PP} can take place as well as its stability properties.

\medskip
The purpose of this paper is to exhibit a family of solution whose soliton resolution is made out of least energy steady states, all centered
at a single point, thus exhibiting multiple blow-up at distinct rates in the form of a ``tower of bubbles''. The solutions we build here are presumably the unique soliton resolutions
possible in the radial case, but this is not known. We analyze stablity of this phenomenon, establishing its universality under small finite energy perturbations.

\medskip
 We recall that all
 positive entire solutions of the equation
 are given by the family of {\em Aubin-Talenti solitons}
 \be
 U_{\mu,\xi} (x)  =   \mu^{-\frac {n-2}2} U\left (\frac {x-\xi}{\mu}   \right )
 \label{bubbles} \ee
 where $U(y)$ is the {\em standard bubble soliton}
 \be\label{U}
 U(y)  =   \alpha_n \left ( \frac 1{1+|y|^2}\right)^{\frac{n-2}2} , \quad \alpha_n = (n(n-2))^{\frac 1{n-2}}.
 \ee
The main characteristic of the critical exponent $p$ is that the energy functional is invariant under the scalings
$u_\mu(x) = \mu^{-\frac{n-2}2} u(\mu^{-1}x)$. In particular we have that
$ J(U_{\mu,\xi}) = J(U)=: S_n$.
The functions $U_{\mu,\xi}$ are steady states of \equ{F}.  In fact
  $\pm  U_{\mu,\xi}$  are precisely the least energy nontrivial solutions of \equ{Y}.
  The only radial solutions of \equ{Y} are given by the functions  $\pm  U_{\mu,0}$. In particular for a radially symmetric solution of \equ{F},
  decomposition \equ{PP} would read as a ``tower of bubbles'' of the form
\be \label{PP1} u(x,t)  =   \sum_{j=1}^k  \frac {\sigma_j}{\mu_j(t)^{\frac{n-2}2}} U \left (\frac{x } {\mu_j(t)}  \right)   + o(1)\ass t\to +\infty.   \ee
where  $\sigma_j\in \{-1,+1\}$ and $\mu_k(t)\ll \cdots\ll \mu_1(t)$. In fact, H. Matano and F. Merle have obtained that in \equ{PP1} signs are alternate: either  $\sigma_j = (-1)^j$ for all $j$ or $= (-1)^{j-1}$ for all $j$ \cite{merle}.

\medskip
In this paper we construct for each given $k\ge 2$ a solution  of \equ{F} with profile decomposition \equ{PP1} and $\sigma_j= (-1)^{j-1}$.
and analyze its stability in the class of all non-radial functions.


\begin{theorem} \label{teo1} Let $n\ge 7$, $k\ge 1$. There exists
a radially symmetric initial condition $u_0(x)$ such that the solution of Problem $\equ{F}$ blows-up in infinite time  exactly at $0$
with a profile of the form
\be
u(x,t) =  \sum_{j=1}^k  (-1)^{j-1} {\mu_j^{-\frac {n-2}2}} U \left( \frac {x}{\mu_j} \right)\,  +\,  o(1) \ass t\to +\infty
\label{forma}\ee
and for certain positive numbers $\beta_j$, $j=1,\ldots, k$ we have
\be
 \mu_j(t)\ =\  \beta_ j t^{- \alpha_j }\, (1+ o(1))  ,
\label{aaaa}\ee
where
$$\, \alpha_j  \, =\, \frac 12  \left( {n-2 \over n-6}\right)^{j-1} -\frac 12  , \quad j=1,\ldots, k.  $$

\end{theorem}
Radial symmetry is not necessary in our construction.
In fact, a by-product of the proof is a {\em codimension $k+ n (k-1)$}-stability result:
the bubble-tower phenomenon persists for initial data chosen in a {\em codimension $k+ n (k-1)$}-manifold of finite-energy perturbations.
Let $u_0(x)$ be the radial initial condition for the solution in Theorem \ref{teo1}.
There exist  smooth,  compactly supported radial functions
$\omega_\ell (x), \ \ell =1,\ldots , \  N_k =k+ n (k-1)  $
 such that the following property holds.

\begin{corollary} \label{coro1}
Let $\alpha > \frac{n-2}2$.
Then for any sufficiently small $\delta>0$ and any $n$-symmetric function
function $z_*(x)$    such that
\be\label{decay}  |z_*(x)| \ \le \   \frac \delta {1 +|x|^\alpha}   \ee
there exist $N_k$ scalars $c_\ell  (z_*)$ and a point $q= q(z_*) \in \R^n$ with $|c_i| + |q|= O(\delta)$,
such that the solution $u(x,t)$ of problem $\equ{F}$ with initial condition
\be u(x,0)\, = \, u_0(x)\, +\, z_*(x)  \,+\, \sum_{\ell =1}^{N_k} c_\ell  (z_*)\, \omega_\ell  (x)  \label{u1}\ee
is globally defined in time and has the  expansion $\equ{forma}$ where $\mu_j(t)$ is as in $\equ{aaaa}$.

\end{corollary}
Thus in order for a small perturbation of $u_0$ to lead to a $k$-bubble tower as $t\to +\infty$  it should
satisfy the $k+ n (k-1) $ scalar constraints $c_\ell (z_*) = 0$. These constraints define $C^1$-functionals in the natural topology for $z_*$ with linearly independent differentials at $z_*=0$, and hence a local $C^1$ manifold with codimension $k+ n (k-1)$  (see Remark \ref{rem}).
   Condition  $\alpha > \frac {n-2}2$ in \equ{decay} is
 sharp to obtain finite energy of the initial condition \equ{u1}, namely
$J(u(\cdot ,0)) <+\infty$. (Finite-time blow-up is expected when $\alpha \leq \frac{n-2}{2}$. See \cite{fk}.)

\medskip
 Theorem \ref{teo1} for $k=1$ is actually trivial, just taking $$u(x,t)= U(x).$$  We observe that corollary \ref{coro1} essentially recovers the 1-codimensional stability of steady states in the energy space established for $n\ge 7$ in \cite{cmr}.

\medskip
  The formal analysis in \cite{fk} yields that no infinite time blow-up of positive solutions of \equ{F} should be present in dimensions $n=5$ or higher, while it should be possible for $n=3,4$. We rigorously established this for $n=3$ in \cite{dmw}. Blow-up by bubbling in finite time for \equ{F} was formally analyzed in \cite{fhv} and rigorous constructions achieved in
 \cite{schweyer,dmw1}. Global unbounded solutions like in Theorem \ref{teo1} are regarded as ``threshold solutions'' for the dynamics of \equ{F}. We refer to \cite{fk, PY1, PY2, qs} and references therein.

 \medskip
 The solutions built in this paper seem to be the first examples of {\em multiple blow-up at a single point} in Problem \equ{F}.
  Related phenomena has been detected in the elliptic Brezis-Nirenberg problem  $\Delta u + |u|^{\frac 4{n-2}}u + \la u  =0$ in a ball. See
  \cite{iv1,iv2} and also \cite{ddm,gjp,mp} for multiple bubbling in the slightly supercritical case. In the parabolic critical setting a construction of ancient solutions with multiple blow-up backwards in time for the Yamabe flow was achieved in \cite{dds}.

\medskip
Blow-up by bubbling (time dependent, energy invariant, asymptotically singular scalings of steady states) is a phenomenon that arises in various problems of parabolic and dispersive nature. This has been an intensively studied topic in the harmonic map flow \cite{LW} and critical heat equations \cite{fhv, dmw1, mm1, mm2, qs}.
Profile decompositions of the type \equ{PP} are well-known to arise in dispersive contexts such as the energy critical wave equation
 \cite{BG, dkm, dkm2,dkm3, KS1}.  The actual classification problem  is known as the {\em soliton resolution conjecture}, see \cite{Cote, DJKM, KNS, Schuur, Tao1}. See also  \cite{DK,kenigmerle,kst, KS2, MR, RR} for related results. 
Constructions of two-bubble solutions in wave and Schr\"odinger energy critical equations under radial symmetry have been recently achieved in \cite{j2,j3,j4,j5}.

\medskip
The method of this paper is close in spirit to the analysis in the works
\cite{cdm,ddw,dmw,dmw1,dmw2}, where the {\em inner-outer gluing method} is employed. That approach consists of reducing the original problem to solving a basically uncoupled system, which depends in subtle ways on the parameter choices (which are governed by relatively simple ODE systems). The new challenge in this paper is to deal with drastic differences in blow-up rates at the same place. A novel topology for both inner and outer problems is introduced. See the norms (\ref{norm1*}) and (\ref{cotaphi}) below. The analysis of ``bubble towers'' in this paper may be useful in  energy critical geometric equations such as the harmonic map flow \cite{LW, Topping, vander} and also in dispersive settings like those mentioned above.

\medskip
Finally, we mention that the problem of blow-up in finite or infinite time for more general $p>1$ in \equ{F} is a classical subject after the seminal work by Fujita \cite{F}. Various different scenarios have been discovered or discarded in the supercritical case. See for instance \cite{collot4,collot1,dmw2,hv,hv1,mm2, MS} and the book \cite{qs}.

\medskip
The rest of this paper will be devoted to the proofs of the above results. We will simultaneously prove Theorem \ref{teo1} and Corollary \ref{coro1}. In \S \ref{s2} we will build a sufficiently accurate
 first approximation for the solution and estimate the error of approximation. In \S \ref{s3} we formulate an ansatz for
 the solution and the {\em inner-outer gluing system} for its unknown. In \S \ref{s4} we discuss the necessary linear theories, and finally solve the problem by means of a fixed point argument in \S \ref{s5}.

\medskip
Throughout this paper,  $\chi(s)$ will denote a smooth cut-off function such that \be \label{cut}
\chi (s)\, =\,  \begin{cases}1 & \hbox{ for }  s\leq 1, \\
0 & \hbox{ for }  s\geq 2 \end{cases} \ee
and, for a set $\Omega \subset \R^n$, ${\bf 1}_\Omega$ will denote the characteristic function defined as
\be \label{funo}
{\bf 1}_{\Omega} (x) \, = \,
\begin{cases} 1 & \hbox{ for }  x \in \Omega, \\
0 & \hbox{ for }  x \in \R^n \setminus \Omega \end{cases} .\ee
 \section{A first approximation and the ansatz}\label{s2}

In what follows we write  Problem \equ{F}  in the equivalent form

\be\left\{
\begin{aligned}
S[u]   \,:=\,  -u_t+ \Delta u+ f(u)  &= 0 \ \, \inn \R^n \times (t_0, \infty),  \\
u(\cdot,t_0) &  =u_0 \inn \R^n
\end{aligned}\right.
\label{F1}\ee
where
$$ f(u)=
|u|^{p-1}u= |u|^{\frac 4{n-2}} u $$ and the initial time $t_0>0$ is left as a parameter which will be later on taken sufficiently large.
The difference is just convenient cosmetics, since then the function $u(x,t+t_0)$ will solve the original problem \equ{F}.
We thus look for a solution $u(x,t)$ of \equ{F1} which looks like a tower of bubbles of the form \equ{forma} centered at $q_0=0$ as $t\to +\infty$.

\medskip
Let us consider $k\ge 2$,  $k$ positive functions
$$\mu_{k}(t) <  \mu_{k-1}(t) <\cdots  <   \mu_1(t)  \inn (t_0,\infty) $$
which will later be chosen, such that as $t\to +\infty,$
\be \label{defpar}
\mu_1(t)\to 1 , \quad
 \frac{\mu_{j+1}(t)}{\mu_{j}(t)}\to  0 \foral j=1,\ldots,k-1  .
\ee
Let us also consider $k$ points $\xi_j$, such that as $t \to +\infty$
\be \label{defpoint}
{|\xi_j (t)| \over \mu_j (t)} \to 0, \quad j=1, \ldots ,k.
\ee
Let us observe that these assumptions on $\mu_j (t)$ and $\xi_j (t)$ imply that
$$
\xi_1 (t) \to 0, \quad {\xi_{j+1} (t) - \xi_{j} (t) \over \mu_{j} (t) } \to 0 \foral j=1, \ldots ,k-1,
$$
a fact that will be used later on.
We denote in what follows
$$\vec \mu = (\mu_1,\ldots,\mu_k) \quad {\mbox {and}} \quad \vec \xi = (\xi_1 , \ldots , \xi_k) .
$$
Let us set
$
\bar \chi (x , t ) = \chi  \left({|x| \over \sqrt{t}} \right)
$
with $\chi(s)$ as in \equ{cut}.
Consistently with \equ{forma}, we write
\be\label{u*}
\bar U  \, = \,   
\bar \chi  \,  \sum_{j=1}^k U_j
\ee
where
$$ U_j (x,t) \, = \,  \frac{ (-1)^{j-1}} { \mu_j  (t)^{{n-2 \over 2}} }\, U\left({ x -\xi_j (t) \over \mu_j (t) } \right)
$$
and $U(y)$ is given by \equ{U}.
 We will get an accurate first approximation to a solution of \equ{F1} of the form  $\bar U  + \varphi_0$  that reduces the part of the error $S[\bar U ]$ created by
the interaction of the bubbles $U_j$ and $U_{j-1}$, $j=2,\ldots, k$. To get the correction $\vp_0$  we will need to fix
the parameters $\mu_j$ at main order around certain explicit values.

\medskip 	
Let us consider the geometric averages
$$
\bar \mu_j :=  \sqrt{\mu_j \mu_{j-1}}, \quad j=2,\ldots, k
$$
  and introduce  the cut-off functions
\be \label{defchij} \chi_j (x, t) \, =\, \left\{
\begin{aligned}
& \chi \left( {2 |x -\xi_j (t) | \over \bar \mu_j } \right) -\chi \left( { |x- \xi_j (t) | \over 2 \bar \mu_{j+1} } \right)  &j&=2, \ldots  , k-1, \\
& \chi \left( {2 |x- \xi_j (t) | \over \bar \mu_k } \right)  \  &j&=k.
\end{aligned}\right.
\ee
We observe that they have the property that
$$
\chi_j (x, t)\, =\, \left\{
\begin{aligned}
0  & \ \hbox{ if }\   \qquad\qquad |x - \xi_j (t)|\le   2 \bar \mu_{j+1},    \\
1 &  \  \hbox{ if }\   \ 4 \bar \mu_{j+1}\le   |x - \xi_j (t)|\le   \frac 12 \bar \mu_{j},\\
0  &  \  \hbox{ if }  \ \qquad\qquad |x - \xi_j (t)| \ge    \bar\mu_{j},
\end{aligned}\right.
$$
with the convention  
$\bar \mu_{k+1} = 0$.
We look for a correction $\vp_0$ of the form
\be \label{defvarphi0}
\varphi_0  =  \sum_{j=2}^k \varphi_{0j}  \chi_j  ,
\ee
where
$$
\varphi_{0j} (x,t ) = {(-1)^{j-1} \over \mu_j(t)^{n-2 \over 2} } \, \phi_{0j} \left({x - \xi_j (t)\over \mu_j (t) },t \right)
$$
for certain functions $\phi_{0j} (y,t)$ defined in entire $y\in \R^n$ which we will suitably determine.
Let us write
\begin{equation}\label{e0}
\begin{aligned}
S(\bar U+ \varphi_0 ) &= S(\bar U ) + {\mathcal L}_{\bar U}  [\varphi_0] + N_{\bar U} [\varphi_0 ]
\end{aligned}
\end{equation}
where
$$ 
 \begin{aligned}{\mathcal L}_{\bar U} [\varphi_0 ] = &-\pp_t \varphi_0 + \Delta_x \varphi_0 +  f'(\bar U) \varphi_0, \\
 N_{\bar U} [\varphi_0 ] = &f(\bar U +\vp_0) - f'(\bar U) \varphi_0 - f(\bar U).
\end{aligned}
$$
Using the homogeneity of the function $f$, we observe that
\be \label{e1}
\begin{aligned}
S(\bar U  ) &= -\sum_{j=1}^k \pp_t (\bar \chi   U_j )   + \bar \chi ^p \, f(\sum_{j=1}^k U_j ) - \bar \chi  \sum_{j=1}^k f(  U_j )   \\
&+ (\Delta_x \bar \chi) (\sum_{j=1}^k U_j ) + 2 (\nabla_x \bar \chi ) (\sum_{j=1}^k \nabla_x U_j)\\
&= \bar E_1 + \bar E_2
\end{aligned}
\ee
where
$$
\begin{aligned}
\bar E_1&= \bar \chi \left[ -\sum_{j=1}^k ( \pp_t U_j )  + f (\sum_{j=1}^k U_j ) - \sum_{j=1}^k U_j \right] \\
\bar E_2&= \left( \bar \chi^p - \bar \chi \right) f (\sum_{j=1}^k U_j ) + ( \Delta_x - \pp_t  ) (\bar \chi ) (\sum_{j=1}^k U_j ) + 2 (\nabla_x \bar \chi ) (\sum_{j=1}^k \nabla_x U_j ).
\end{aligned}
$$
We decompose $\bar E_1 $ in different regions defined by the cut-off functions $\chi_j$ introduced in \eqref{defchij} as follows
\begin{equation}\label{e11}
\begin{aligned}
\bar E_1 &= -(\pp_t U_1 ) \bar \chi + \sum_{j=2}^k \left[ -\pp_t  U_j + f' (U_j) U_{j-1} (0) \right] \chi_j + \bar E_{11}  ,
\end{aligned}
\end{equation}
where
\begin{equation}\label{e11n}
\begin{aligned}
\bar E_{11} &= \sum_{j=2}^k \left[ f' (U_j ) (\sum_{l\not= j, j-1} U_l ) + f' (U_j)  \big( U_{j-1} - U_{j-1}(0) \big) \right] \chi_j  \\
&+\sum_{j=2}^k \left[ N_{U_j} \big(\sum_{l\not= j, j-1} U_l \big) -\sum_{l\not= j} f( U_l ) \right] \chi_j   - \bar \chi \sum_{j=2}^k (1-\chi_j ) \pp_t U_j \\
&+  \bar \chi  \big[  f\big(\sum_{j=1}^k U_j \big)   - \sum_{j=1}^k  f (U_j)  \big]  \big(1-  \sum_{l=2}^k \chi_l \big)  ,
\end{aligned}
\end{equation}
with
$$N_{U_j} \big(\sum_{l\not= j, j-1} U_l \big) = f\big(\sum_{l=1}^k U_l \big) - f(U_j ) - f' (U_j) \big(\sum_{l\not= j} U_l \big) . $$

Next we write ${\mathcal L}_{\bar U } [\varphi_0]$ using the form of $\varphi_0$ in \eqref{defvarphi0} as follows
\begin{equation}\label{e2}
\begin{aligned}
  {\mathcal L}_{\bar U } [\varphi_0]&=  \sum_{j=2}^k \left[ \Delta_x \varphi_{0j} + f'(U_j) \varphi_{0j} \right] \chi_j  \\
  &+ \sum_{j=2}^k p (f'(\bar U) - f'(U_j) ) \varphi_{0j} \chi_j   + \sum_{j=2}^k \left[ 2 \nabla_x \varphi_{0j} \nabla_x (\chi_j ) + \Delta_x (\chi_j )  \varphi_{0j} \right] \\
  &- \sum_{j=2}^k \pp_t (\varphi_{0j} \chi_j ).
  \end{aligned}
  \end{equation}
Replacing \eqref{e1}, \eqref{e11}, \eqref{e2} into \eqref{e0}, and reorganizing properly the terms, we obtain
\begin{equation}\label{e3}
\begin{aligned}
S(\bar U+ \varphi_0) &= - \bar \chi   \pp_t U_1  +\sum_{j=2}^k \left[ \Delta_x \varphi_{0j} +  f'(U_j) \varphi_{0j} -\pp_t U_j + f' (U_j) U_{j-1} (0) \right] \chi_j  \\
& + \bar E_{11} + \bar E_2  + \sum_{j=2}^k p (f'(\bar U)- f'(U_j) ) \varphi_{0j} \chi_j   \\
& + \sum_{j=2}^k \left[ 2 \nabla_x \varphi_{0j} \nabla_x (\chi_j  ) + \Delta_x (\chi_j  ) \varphi_{0j} \right]
  - \sum_{j=2}^k \pp_t (\varphi_{0j} \chi_j  )  +  N_{\bar U} [\varphi_0 ], \\
\end{aligned}
\end{equation}
where $\bar E_{11}$ and $\bar E_2$ are defined respectively in \eqref{e11n} and \eqref{e1}.
The functions $\varphi_{0j}$ will be chosen to  eliminate at main order the terms in the first line of \equ{e3}, after conveniently restricting the range of variation of $\mu$ and $\vec \xi$,
\begin{equation}\label{e4}
\begin{aligned}
E_j [\vp_{0j}; \vec\mu, \vec \xi ] \ :=&\ \Delta_x \varphi_{0j} + f'(U_j) \varphi_{0j} -\pp_t U_j + f' (U_j) U_{j-1} (0) \\
=&\ {(-1)^{j-1} \over \mu_j^{n+2 \over 2}} \Big[ \Delta_y \phi_{0j} + p U (y)^{p-1} \phi_{0j} + \mu_j \dot \mu_j Z (y)  \\
&\ - p U^{p-1} (y) \left( \frac{\mu_j}{  \mu_{j-1} } \right)^{n-2 \over 2} U(0) +\mu_j \dot \xi_j \cdot \nabla U(y) \Big]_{y= \frac {x- \xi_j (t) }{\mu_j}}
\end{aligned}
\end{equation}
where 
$Z_{n+1} (y) = {n-2 \over 2} U(y) + y\cdot \nabla U(y). $
The elliptic equation (for a radially symmetric function $\phi(y)$)
\be\label{ss}
\Delta_y \phi  + p U(y)^{p-1}  \phi  + h_j(y,\mu) = 0  \inn \R^n
\ee
where $$h_j(y,\mu) = \mu_j \dot \mu_j Z_{n+1}  (y) -
p U (y)^{p-1} \left( \frac{\mu_j}{  \mu_{j-1} } \right)^{n-2 \over 2} U(0)$$
has a solution with $ \phi(y)\to 0 \ass |y|\to \infty$ if and only if $h_j$ satisfies the solvability condition
 $$\int_{\R^n} h_j(y,\mu)Z_{n+1} (y)\, dy = 0 .$$
The latter conditions hold if the parameters $\mu_j(t)$ satisfy  the following relations:
\be \label{systemmu0}
\mu_1=1, \quad  \mu_{j} \dot \mu_{j} =-c \la_j ^{n-2 \over 2} ,\quad \la_j = {\mu_j \over \mu_{j-1} } \foral   j=2, \ldots , k,
\ee
where
\be \label{defc}
c\,= \, - U(0) {p \int_{\R^n} U^{p-1} Z_{n+1}  \, dy \over \int_{\R^n} Z_{n+1}^2 dy} \, = \,   U(0)\frac{n-2}2 {  \int_{\R^n} U^p \, dy \over \int_{\R^n} Z_{n+1} ^2 dy} >0 .
\ee
 We  let  $\vec\mu_0 = ( \mu_{01}, \ldots \mu_{0k} )$ be the  solution of \equ{systemmu0} in $(t_0,\infty)$
given by
\be \label{mu0}
\mu_{0j} (t) =\beta_j t^{-\alpha_j} , \quad  t\in (t_0,\infty)
\ee where
$$
\alpha_j = {1\over 2} \left ({n-2 \over n-6} \right)^{j-1} - {1\over 2}, \quad j=1,\ldots, k
$$
and the numbers $\beta_j$ are determined by the recursive relations
$$
\beta_1 = 1, \quad  \beta_j =  \left({n-2 \over n-6} \alpha_{j-1}+ { 2 \over n-6} \right)^2 \,  \beta_{j-1}^{n-2 \over n-6}. \quad
$$
From \eqref{systemmu0}, we see that setting
$$
\la_{0j}(t) = {\mu_{0j} \over \mu_{0,j-1}}(t)
$$
we have
$$
 h_j(y,\mu_0) =
      \la_{0j}^{n-2 \over 2}\bar h(y), \quad   \bar h(|y|)=    c p U(0) U (y)^{p-1} +Z_{n+1}(y).
 $$
 Since $\int_{\R^n} \bar h Z_{n+1}  \, dy = 0$,  there exists a radially symmetric solution
 $\bar \phi(y)$ to the equation
$$
 \Delta \bar \phi + p U(y)^{p-1}  \bar \phi +\bar h(|y|) =0  \quad {\mbox {in}} \quad \R^n.
$$
 such that $ \bar\phi (y) = O(|y|^{-2})$ as $|y|\to +\infty$.
 Indeed, writing with some abuse of notation  $\bar\phi(y) = \bar\phi(|y|)$ the above equation becomes
 \be\label{k1}
 \mathcal L [\bar\phi] :=  \bar\phi''(\rho)  +  \frac {n-1}\rho \bar\phi'(\rho)   = - \bar h(\rho) , \quad \rho\in (0,\infty)
 \ee
We observe that   $\mathcal L [Z_{n+1} ]=0$ and that there is a second linearly independent $\ttt Z(\rho)$ with $\mathcal L [\ttt Z]=0$, with
$Z(\rho) = O(\rho^{2-n}) $ as $\rho\to 0$ and $\ttt Z(\rho) =O(1) $ as $\rho \to +\infty$, which we can choose so that
the variation of parameters formula
\be\label{formula}
\bar\phi(\rho)   =  \ttt Z(\rho)  \int_\rho^\infty\bar h(r) Z_{n+1} (r) r^{n-1}dr  +     Z_{n+1} (\rho)  \int_0^\rho \bar h(r) \ttt Z(r) r^{n-1}dr
\ee
gives a solution of \equ{k1}. Since $ \int_0^\infty\bar h(r) Z_{n+1} (r) r^{n-1}dr = 0 $ the above solution is regular at the origin and satisfies
$ \bar\phi (\rho) = O(\rho^{-2})$ as $\rho \to +\infty$.

\medskip
 Then we define $\phi_{0j}(y,t)$ as
 \be \label{defphi0j}
\phi_{0j} (y,t) =    \la_{0j}^{n-2 \over 2} \, \bar \phi (y).
\ee
Thus $\phi_{0j}$ solves equation \equ{ss}.

\medskip
In what follows we let the parameters $\mu_j(t)$  in \eqref{defpar} have the  form
$\vec \mu = \vec \mu_0 + \vec \mu_1 $ or
\be \label{muj0}
\mu_j(t) = \mu_{0j}(t) + \mu_{1j}(t), \quad
\ee
where the parameters  $\mu_{1j}(t)$ to be determined satisfy for some small and fixed $\sigma>0$
\be\label{assmu}
\mu_{0j}|\dot \mu_{1j} (t)| \  \le\    \la_{0j}(t)^{\frac{n-2}2 } t^{-\sigma}  , \quad
\ee
Condition \equ{assmu} implies
$$
\lim_{t \to \infty} {\mu_{1j} (t)  \over \mu_{0j} (t)  } = 0.
$$
We will also assume that the points $\xi_j$ in \eqref{defpoint} satisfy
$$
\mu_{0j}|\dot \xi_{j} (t)| \  \le\    \la_{0j}(t)^{\frac{n-2}2 } t^{-\sigma}  .
$$
It is  convenient to write
$$
\la_j(t) = {\mu_j \over \mu_{j-1} }(t) = \la_{0j}(t) + \la_{1j}(t), \quad j=2, \ldots , k.
$$
We observe that for some positive number $c_j$ we have
$$
\la_{0j} (t) =  c_jt^{-  {2\over n-6} \left({n-2 \over n-6}\right)^{j-2} }.
$$
With these choices, we have that $E_j[\vp_{0j}; \vec\mu_0, \vec 0 ] =0 $. The expression $E_j[\vp_{0j}; \vec \mu , \vec \xi ]$ in \eqref{e4} can be decomposed as
$$
\begin{aligned}
E_j[\vp_{0j}; \vec \mu_0 +\vec \mu_1, \vec \xi ] =&  \left[  \mu_j \dot \mu_j  - \mu_{0j} \dot \mu_{0j} \right]  Z_{n+1} (y_j)  -
p U^{p-1} (y_j) \left[ \la_j^{n-2 \over 2} - \la_{0j}^{n-2 \over 2} \right]  U(0) \\
+& \mu_j \dot \xi_j \cdot \nabla U(y_j) \\
=& \mu_j^{-\frac{n+2}2} D_j[\vec \mu_1  ]  +   \Theta_j [\vec \mu_1,\vec \xi], \quad   y_j =  { x -\xi_j (t) \over \mu_j (t) }  
 \end{aligned}$$
where for $j=2, \ldots , k$
\begin{equation}\label{e55}
\begin{aligned}
 D_j[\vec \mu_1  , \vec \xi  ] = & (\dot \mu_{0j} \mu_{1j} + \mu_{0j} \dot \mu_{1j} )  Z_{n+1} (y_j)  +    \frac{n-2}2 p U^{p-1} (y_j) U(0) \la_{0j}^{n-4 \over 2}   {\mu_{1j} \over \mu_{0, j-1} } + \mu_j \dot \xi_j \cdot \nabla U(y) ,\\
\Theta_j[\vec \mu_1,\vec \xi] =& - \pp_t (\mu_{1j}^2) Z_{n+1} (y_j) - p U^{p-1} (y_j) U(0) {n-2 \over 2} \la_{0j}^{n-2 \over 2}{\mu_{1j-1} \over \mu_{0j-1} }  \\
&\ + p U^{p-1} (y_j)\la_{0j}^{n-2 \over 2}\, O\big(  {\mu_{1j} \over \mu_{0j-1} } - {\mu_{1j-1} \over \mu_{0j-1} }\big)^2.
 \end{aligned}
\end{equation}
with this choice of $\mu_{0j}$, we observe that the first term in \eqref{e11n} takes the more explicit form
$$
- \bar \chi  \pp_t U_1 = {\bar \chi \over \mu_1^{n+2 \over 2} } (1+ \mu_{11} ) [ \dot \mu_{11} Z_{n+1} (y_1) + \dot \xi_1 \cdot \nabla U (y_1) ] , \quad y_1 = {x- \xi_1 (t) \over \mu_1 }.
$$
Define
\be \label{D0}
D_1 [\vec \mu_1, \vec \xi ] = \dot (1+ \mu_{11} ) [ \dot \mu_{11} Z_{n+1}(y_1) + \dot \xi_j \cdot \nabla U(y_1) ] , \quad y_1 = {x - \xi_1 (t) \over \mu_1 }.
\ee
We define our approximate solution to be given by $u_* = u_* [\vec \mu_1 , \vec \xi ]$ as
\be \label{defustar}
u_*  = \bar U + \varphi_0
\ee
where $\bar U$ is defined by \eqref{u*} and $\varphi_0$ has the form \eqref{defvarphi0}, with $\phi_{0j}$ defined by \eqref{defphi0j}, and $\mu_j$ defined by \eqref{muj0} and \eqref{mu0}.

\section{The inner-outer gluing system }\label{s3}
We consider the approximation $u_* =u_*[\vec\mu_1 , \vec \xi ]$ in \equ{defustar} built in the previous section and want to find a solution  of equation \equ{F1} in the form $u = u_* +\vp $.
The problem becomes
\be\left\{
\begin{aligned} &S[u_*+\vp] = \\
&-\vp_t  + \Delta \vp  +    f'(u_*) \vp   +    N_{u_*}[\vp] +  S[u_*]  = 0 \inn \R^n \times (t_0, \infty) \\
&\vp(\cdot,t_0)   =\vp_* \inn \R^n.
\end{aligned}\right.
\label{F2}\ee
where
$$
N_{u_*}[\vp]= f(u_*+\vp ) - f'(u_*)\vp - f(u_*)
$$
the function $\vp_*(x)$ is an initial condition to be determined, and  in \equ{F1} we have $u_0 = u_*(\cdot ,0) +  \vp_*   $.

\medskip
We consider the cut-off functions  $\eta_j, \ \zeta_j$, $ j=1,\ldots, k,$ defined as
\be \label{newcuts}\begin{aligned}
 \eta_j (x,t) &=   \chi \left (  \frac{|x - \xi_j (t)|} {  R\mu_j (t)  } \right ) \\
 \zeta_j (x,t) &=   \chi \left (  \frac{|x- \xi_j (t)|} {  R\mu_j (t) } \right )  -    \chi \left ( \frac {|x- \xi_j (t)|} {R^{-1} \mu_j(t)  } \right )
\end{aligned} \ee
We observe that
$$\eta_j(x,t)\ =\ \left\{
\begin{aligned}1 & \ \hbox{ for } \   |x- \xi_j (t) |\le   R\mu_j(t), \\  0& \ \hbox{ for }\
|x- \xi_j (t)|\ge   2R\mu_j(t)  \end{aligned} \right.$$
and
$$\zeta_j(x,t)\ =\ \left\{
\begin{aligned}1 & \ \hbox{ for } \  \qquad\qquad 2R^{-1}\mu_j(t)\le    |x - \xi_j (t)|\le   R\mu_j(t) \\  0& \ \hbox{ for }\
|x |\ge   2R\mu_j(t)\  \hbox{ or }  \  \ |x- \xi_j (t) |\le    R^{-1} \mu_j(t) . \end{aligned} \right.$$
We will in addition choose  an  $R$ to be a $t$-dependent, slowly growing function, say
\be    \quad R(t) = t^\ve , \quad t>t_0 \label{epsilon}\ee
where $\ve>0$ will be later on fixed sufficiently small.

\medskip
We consider functions $\phi_j(y,t)$ $j=1,\ldots, k$ defined for $|y|\le  3R$  and a function $\psi(x,t)$ defined in $\R^n\times (t_0,\infty)$.
We look for a solution $\vp(x,t)$ of \equ{F2} of the form
\be
\vp  =   \sum_{j=1}^k  \vp_j  \eta_j  \, +\  \Psi,
\label{foorma}\ee
where
$$
\vp_j(x,t) \ =\  \frac {(-1)^{j-1}}{\mu_j^{\frac{n-2}2}} \phi_j \left (\frac {x- \xi_j (t) }{\mu_j (t) },t \right).
$$
Let us substitute $\vp$ given by $\equ{foorma}$ into equation \equ{F2}. We get
$$
\begin{aligned} S[u_*  +\vp] =&
  \sum_{j=1}^k \eta_j ( -   \pp_t \vp_j  + \Delta_x \vp_j +  f'(U_j)  \vp_j  + \zeta_j f'(U_j)  \Psi +  \mu_j^{-\frac{n+2}2} D_j  [\vec \mu_1, \vec \xi ]  )  \\
 & -\Psi_t +   \Delta_x \Psi +   V  \Psi   +  B[\vec \phi]  + \mathcal N(\phi,\Psi;\vec \mu, \vec \xi ) +  E^{out}
\end{aligned}
$$
Here we denote for $\vec \phi= (\phi_1,\ldots, \phi_k),$ $\vec \mu= (\mu_1,\ldots, \mu_k)$, $\vec \xi = (\xi_1 , \ldots , \xi_k )$
\be \label{poto}
\begin{aligned}
&B[\vec \phi] = \sum_{j=1}^k  2\nn_x\eta_j \nn_x \vp_j  +  ( -\pp_t \eta_j + \Delta_x \eta_j) \vp_j +
 \sum_{j=1}^k \eta_j (f'(u_*)- f'(U_j))\vp_j\\
 & \quad \quad + \dot \mu_j {\partial \over \partial \mu_j} \varphi_j \eta_j +\dot \xi_j \nabla_{\xi_j} \varphi_j \eta_j \\
 & \mathcal N( \vec \phi , \psi; \vec \mu, \vec \xi )   =     N_{u_*} \big ( \sum_{j=1}^k  \vp_j  \eta_j  \, +\  \psi   \big), \quad
 V =  f'(u_*) -  \sum_{j=1}^k \zeta_j f'(U_j), \\ &E^{out} =   S[u_*] - \sum_{j=1}^k \mu_j^{-\frac{n+2}2} D_j[\vec \mu_1, \vec \xi ] \eta_j  .
\end{aligned}
\ee
where $D_j[\vec \mu_1, \vec \xi ]$ is the operator defined in \equ{e55} and \eqref{D0}.

We will have that $S[u_* + \vp ] = 0 $ if the following system of $k+1$ equations is satisfied.
\be \label{inner}
 -   \mu_j^2 \pp_t \phi_j  + \Delta_y \phi_j +  pU(y)^{p-1}  \phi_j  + \zeta_j U(y)^{p-1}  \mu_j^{\frac{n-2}2} \Psi +   D_j[\vec\mu_1 , \vec \xi ] = 0
\ee
\be \label{outer}
-\Psi_t  + \Delta_x\Psi   +  V \Psi  +  B[\vec\phi] +    \mathcal N ( \phi , \Psi; \vec \mu, \vec \xi)  +   E^{out} = 0
\ee
In the next sections we will find a solution to this system with the appropriate size. We will be able to do that only choosing
properly  the parameters  $\vec\mu_1$ and the points $\vec \xi$. We shall formulate this problem creating a system involving the parameters as a part of the
unknowns.

\section{The linear outer and inner problems}\label{s4}

In order to solve system \equ{inner}-\equ{outer}, in this section we find inverses and corresponding estimates for their main linear parts.

\subsection{The linear outer problem}
In this section we consider the issue of finding estimates through barriers for the unique solution of
\be\label{heat}\left \{
\begin{aligned}
\psi_t \ = &\ \Delta_x \psi  +  g(x,t)   \inn \R^n \times (t_0,\infty)\\
\psi (\cdot,t_0) \ = &\ 0 \inn \R^n
\end{aligned}\right.
\ee
given by Duhamel's formula
\be\label{duh}
\psi(x,t) = \TT^{out}[g] (x,t)=   \frac 1{ (4\pi)^{\frac n2} }\int_{t_0}^t    \frac {ds}{ (t-s)^{\frac n2 } }  \int_{\R^n}    e^{ -\frac {|x-y|^2 }{4(t-s)}} g(y,s)\, dy .
\ee
where $g \in L^\infty (\R^n \times (0,\infty)) $.
The class of right hand sides $g$ that we want consider in this section are  those needed to solve the outer problem.
We begin with a class of right hand sides that are better expressed in selfsimilar form.
Let us consider the function
$$
g_0(x,t) =  \frac 1{t^{d+1}} h\left ( \frac x{\sqrt{t}}  \right)
$$
where $0\le d\le \frac n2 $. We assume that for a positive function $h$ we have
\be
|g(x,t) | \le  g_0(x,t)
\label{ll}
\ee

\begin{lemma} \label{lema1}
There exists a constant $C$ such that for
all $g$ in equation $\equ{heat}$ that satisfies $\equ{ll}$ and the solution $\psi$ given by $\equ{duh}$ we have that
\begin{enumerate}
\item \label{pp}
 If $h$ is compactly supported  then
$$
|\psi (x,t)|  \ \le \ \frac C{t^d} e^{-\frac{|x|^2} {4t}}
$$

\item \label{pp2}
If  for some $m> 2d $ we have
$
h(z) =  \frac 1{ 1+ |z|^m}  ,
$
then
$$
|\psi (x,t)|  \ \le \  \frac { Ct^{\frac m2 -d} } { t^{\frac m2} + |x |^m}
$$

\end{enumerate}

\end{lemma}
\proof
We look for a supersolution of problem \equ{heat}, for $g$ satisfying \equ{ll}  of
the form
$$
\psi(x,t) =  \frac 1{t^{d} }  f\left ( \frac x{\sqrt{t}} \right)
$$
The differential inequality
$$
-\psi_t +  \Delta \psi  + g_0(x,t)  \le  0
$$
is equivalent to the following relation for $f(\xi)$
\be
L_d [f] +  h(\xi) \le 0  \quad \xi\in (0,\infty)
\label{aaa}\ee
where
$$
L_d [f]  = f''(\xi)  +  \frac{n-1} {\xi}{f'(\xi)}   +   d f(\xi)  + \frac 12 \xi f'(\xi)
$$
Let us assume that $d\le \frac n2$ and that $h$ is compactly supported. For $d=\frac n2$ the following function is an exact solution.
$$
f(\xi)=  e^{-\frac {\xi^2} 4}  \int_0^\xi e^{\frac {\rho^2} 4}\rho^{1-n} d\rho  \int_\rho^\infty e^{-\frac {s^2} 4} h(s) s^{n-1}ds
$$
This $f$ is also a positive supersolution to equation \equ{aaa}. Inequality \equ{pp} then immediately follows.

\medskip
Let us now assume that $h(\xi) = \frac 1{ 1+ \xi^m}   $ and choose as a supersolution  for all $\xi$ sufficiently large a function of the form $\bar f(\xi)  = \frac C{\xi^m}$ for a large enough large $C$. In fact for $m> 2d $    we will have $L_d[\bar f] + h  < 0$ for  $\xi \ge M $.  If we consider the usual smooth cut-off function $\chi(s)$ we have then that
\be \label{co} f(\xi) =  (1 - \chi (\xi-M )) \bar f (\xi)  + \bar f_1  (\xi) \ee
 will satisfy  the  differential inequality  $L_d[f] + h\le 0$ in case that
$$\begin{aligned}
L_d[f] + h = & L_d[\bar f_1]  +    2\chi'  \bar f'  +   (\chi''   +  \frac {n-1}{\xi}\chi' +  \frac 12 \xi \chi'  )  \bar f \\ &+
(1-\chi) (L_d[\bar f]  +  h) + \chi h \le    \quad  L_d[\bar f_1]  + h_c(\xi) \le 0
\end{aligned} $$
where the function $h_c$ is compactly supported. Using \equ{pp} we then find a positive supersolution $\bar f_1$   of $L_d[\bar f_1]  + |h_c(\xi)| \le 0 $
with a Gaussian decay. From here and \equ{co},  relation \equ{pp2} readily follows. \qed

\bigskip
Next we consider a class of right hand sides $g$ which satisfy \equ{ll} for a class of functions  $g_0$ which are not of self-similar form.

Let  us consider a positive function
$\la(t)$  such that for some $a>0$
\be \dot \la(t) =   t^{-a-1}  (1 + o(1))  \ass t\to + \infty, \label{ccla}\ee
and a point $\xi (t)$ such that
$$
{ t |\dot \xi (t) |\over \la (t)} =o(1) , \quad {\mbox {as}} \quad t \to +\infty.
$$
For numbers $m>2$ and $\alpha < m$ we  consider the function  $g_1(x,t)$  given by
\be\label{g01}
 g_1(x,t)  =  \frac 1{\la(t)^{2+\alpha} } \frac 1{1+ |y|^{m}} \chi\left (\frac {|x|} {\sqrt{t}}  \right) , \quad y= \frac{x-\xi (t) }{\la(t)} \, .
\ee
We consider again functions $g$ with
\be \label{lll}
|g(x,t)|\ \le \ g_1(x,t) \foral (x,t)\in (t_0, \infty)
\ee
and find suitable barriers for \equ{heat} dependent on whether $m<n$ or $m>n$.

\begin{lemma}\label{lema2}
Let us assume that $m\ne n$.
There exists a $C>0$ such that for all $g$  satisfying
 $\equ{lll}$ with $g_1$ given by $\equ{g01}$, the solution $\psi(x,t)$ of $\equ{heat}$ given by $\equ{duh}$
 satisfies the inequality
 $$
|\psi(x,t)| \, \le \,  C\, \Big [
\frac { \la^{-\alpha}} {1 + |y|^{\bar m -2 }}   +       \frac 1{t^b} e^{-\frac{|x|^2}{ 4t}}  
\Big ],  \quad y= \frac {x-\xi (t) }
{\la(t)}
$$
where $\bar m = \min\{m,n\},$ and  setting $ b= a (\bar m -2 -\alpha)+ \frac {\bar m -2 }2 $, we write
  $$ \quad   \bar b =  \begin{cases} \frac n 2 &\quad{ if } \quad b> \frac n 2   \\ b& \quad{ if }\quad  b< \frac n 2
  \end{cases}   . $$
Moreover, we have the following local estimate on the gradient
$$
|\nabla_x \psi (x,t) | \, \leq \, C \, \la^{-1} R^{-1} \frac { \la^{-\alpha}} {1 + |y|^{\bar m -2 }}, \quad {\mbox {in}} \quad |{x-\xi (t) \over \la (t)} |<R.
$$

\end{lemma}
\proof
We need to find a positive supersolution to
\be\label{su1}
-\psi_t +  \Delta \psi  + g_0(x,t) \le 0 \inn \R^n\times (t_0,\infty).
\ee
Let us consider  the radial solution $P(z)$ of the equation
$$
-\Delta_z P (z) =  \frac 1{1+ |z|^m}  \inn \R^n
$$
given by the formula
$$P(z) =  2 \int_{|z|}^\infty  \frac {dr}{r^{n-1}} \int_0^r   \frac {\rho^{n-1}d\rho} { 1+     \rho^{m}} .$$
Clearly we have
$$
 |z\cdot \nn P(z)| +   P(z)  \ \le \ \frac C{1+|z|^{\bar m -2}}
$$
where $\bar m = \min\{m, n\}$. Let us assume first $m>n$.
With no loss of generality we take $m \le n+ \sigma$ for an arbitrarily small, fixed  $\sigma>0$. Let us define
$$ \bar \psi (x,t) =  \la(t)^{-\alpha} P\left ( \frac {x-\xi (t) }{\la(t)}    \right ) $$
and compute for $| x| < t^{\frac 12}$, using assumptions \equ{ccla},
$$\begin{aligned}
 E(x,t)\ :=&\ - \bar \psi_t  + \Delta \bar \psi  +  g_1(x,t) \,\\ =&\  \,
  \la^{-\alpha} \Big [ \frac {\dot \la}{\la}
   y\cdot \nn_y P(y) + \alpha \frac {\dot \la}{\la}  P(y)  + \nabla_y P(y) \cdot {\dot \xi \over \la}  -  \frac 1{\la^2} \frac 1{ 1+     |y |^{n+\sigma}}\Big ] \\
< &\  \frac {\la^{-\alpha}}{1 + |y|^{n-2}}\Big [  \frac c t \,  -\,   \frac 1{\la^2}\frac 1{ 1+     |y |^{2+\sigma}}       \Big ] \, ,\quad y= \frac {x-\xi(t)}{\la(t)}.
\end{aligned}$$
Reducing the value of $\sigma$ if necessary, we get that the above quantity is negative provided that $|x-\xi | <  t^{\frac 12 -\ve}$, with an arbitrarily small $\ve>0$.  
Now, for $|x-\xi |> t^{\frac 12 -\ve}$ we have
$$\begin{aligned}
E(x,t) \le  &\  \frac  c t   \frac { \la^{n -2 -\alpha } }{ |x-\xi |^{n  -2}}\\   \le  &\
\frac c{t^{1+  (n -2 -\alpha)a +\frac{n-2}2}}  \frac 1{    | t^{-\frac 12}  (x-\xi )|^{n -2}  + t^{-(n-2)\ve}   } \\ \le  & \
\frac c{t^{1+  (n -2 -\alpha)a +\frac{n-2}2 -(n-2)\ve }}  \frac 1{    | t^{-\frac 12}  (x-\xi )|^{n -2}  + 1   }\\ \le  &\
\frac c{t^{1+b}}   \frac 1{    | t^{-\frac 12}  (x-\xi )|^{n-2}  + 1   }\, .
\end{aligned}$$
Now, we write
$
\psi    = \chi((x-\xi )/\sqrt{t}) \bar \psi + \bar\psi_1
$.
Then relation \equ{su1} amounts to finding a positive  $\bar\psi_1$
satisfying
$$
- \pp_t \bar\psi_1  + \Delta \bar\psi_1  + \frac 1{t^{1+b}} h((x-\xi)/\sqrt{t})
$$
for a certain smooth, compactly supported $h(\xi)$. Then Lemma \ref{lema1} provides a positive supersolution of the type
$$
\bar\psi_1(x,t) = \frac C{t^{1+\bar b}} e^{-\frac{|x-\xi|^2} {4t}} .
$$
The proof is concluded in the case $m>n$. The proof for $m<n$ is the same, in fact slightly simpler since we can directly
take in the above argument  $\sigma=\ve =0$.

\medskip
To get the local estimate for the gradient, we write $\psi (x,t) = \tilde \psi ({x-\xi \over \la} , \tau)$, with ${d \tau \over dt} = \la^{-2} (t)$, and observe that
$$
\tilde \psi_\tau = \Delta_z \tilde \psi +\tilde g (z,\tau ), \quad |\tilde g (z,\tau ) | \leq {\la^{-\alpha} \over 1+ |z|^m} \, \chi ({\la z \over \sqrt{t} } ).
$$
We have already established that
$$
|\tilde \psi (z,\tau ) | \leq C {\la^{-\alpha} \over 1+|z|^{m-2} } , \quad |z| <R.
$$
Standard parabolic estimates give, for $\tau_1 > \tau (t_0)$, for fixed $M>0$,
$$
\begin{aligned}
\| \nabla_z \tilde \psi (\cdot , \tau_1) \|_{L^\infty (B_M (0))} \, & \, \leq C \left[ \| \tilde \psi \|_{L^{\infty} (B_{2M} (0) \times (\tau_1 - 1, \tau_1 ))} +  \| \tilde g \|_{L^{\infty} (B_{2M} (0) \times (\tau_1 - 1, \tau_1 ))}\right]\\
&\, \leq C \la^{-\alpha}.
\end{aligned}
$$
In the original variables, we get for $t >t_0$,
$$
R \la | \nabla_x \psi (x,t) | \leq C \frac { \la^{-\alpha}} {1 + |y|^{\bar m -2 }}, \quad {\mbox {for}} \quad |{x-\xi \over \la } |<R.
$$

\qed

\medskip
We apply the results above to derive estimates of the solution of \equ{heat} for a right hand side $g$ controlled by several different weights, that are designed ad-hoc to treat the outer problem \eqref{outer}.
Let us fix $\sigma>0$,  $a>0$ and $\beta$ with
$$
0<a<n-2, \quad 2< \beta < n
$$
and
 define the following weights:
\be \label{omega1}\left\{ \begin{aligned}
\omega_{11} (x,t)\ =& \ { t^{-1-\sigma}  \over (1+ |x-\xi_1  | )^{2+a} }\,  \chi
\left( { |x-\xi_1  | \over \sqrt{t}} \right) \\  \omega_{11}^* (x,t)\ =&\ { t^{-1-\sigma} \over (1+ |x-\xi_1  | )^{a } }\,  \chi
\left( { |x -\xi_1 | \over \sqrt{t}} \right)
\end{aligned}\right.\ee
and for $2\le j\le k$,
\be \label{omega1j}\left\{ \begin{aligned}
\omega_{1j} (x,t)\ =&  \   {t^{-\sigma } \over \mu_j^{n+2 \over 2}}  \, { {\la_j^{{n-2 \over 2} } \over (1+ |{x -\xi_j \over \mu_j} | )^{2+a} }}\,  \chi
\left( { |x-\xi_j | \over \bar \mu_j} \right) \\  \omega_{1j}^* (x,t) \ =&  \    {t^{-\sigma } \over \mu_j^{n-2 \over 2}}  \, { {\la_j^{{n-2 \over 2} } \over (1+ |{x -\xi_j \over \mu_j} | )^{a } }}\,  \chi
\left( { |x-\xi_j| \over \bar \mu_j} \right)
\end{aligned}\right.\ee
\be \label{omega2j} \left\{ \begin{aligned}
\omega_{2j} (x,t) \ =&  \,  { t^{-\sigma} \over \bar \mu_j^{n+2 \over 2}}  \, { \la_j^{{n-2 \over 4} } \over (1+ |{x -\xi_j  \over \bar \mu_j} | )^{n } } \\  \omega_{2j}^* (x,t)\ =&  \  {t^{-\sigma} \over \bar \mu_j^{n-2 \over 2}}  \, { \la_j^{{n-2 \over 4} } \over (1+ |{x -\xi_j \over \bar \mu_j} | )^{n-2 } }
\end{aligned}\right.\ee
and
\be \label{omega3}\left\{ \begin{aligned}
\omega_3 (x,t)\ =&\  {1\over \left( \sqrt{t} + |x| \right)^{\beta}} \\
\omega^*_3 (x,t)\  =&\  {1\over \left( \sqrt{t} + |x| \right)^{\beta -2}}
\end{aligned}\right. .\ee

We claim that the following estimates hold
\begin{prop}\label{merda}

Let us consider $g$ and $\psi$ as in $\equ{heat}$ and $\equ{duh}$.
 There exists a $C>0$ such that:

\begin{enumerate}

\item\label{a1}
 If  \  $
|g(x,t)| \le  \omega_{11}(x,t) $ \ then
$$
\quad |\psi(x,t) | \ \le\  C \Big ( \omega_{11}^* (x,t) + \frac 1{t^{ 1+ \frac a2 + \sigma}} e^{-\frac{|x-\xi_1 |^2}{4t}} \Big ).$$

\medskip
\item\label{a2}
If  for $j\ge 2$  \  $|g(x,t)| \le  \omega_{1j}(x,t) $ \ then
 $$ \qquad\quad |\psi(x,t)  |\  \le\  C \Big (\omega_{1j}^*(x,t) +  \omega_{2j}^*(x,t)  +  \frac 1{t^{\frac n2}} e^{-\frac{|x-\xi_j |^2}{4t}} \Big ) .$$

\medskip
\item\label{a3}
If  for $j\ge 2$  \  $|g(x,t)| \le  \omega_{2j}(x,t) $ \ then
 $$ |\psi(x,t)  |\  \le\  C \Big (\omega_{2j}^*(x,t)   +  \frac 1{t^{\frac n2}} e^{-\frac{|x-\xi_j |^2}{4t}} \Big ).\qquad $$

\medskip
\item\label{a4}
If  \  $
|g (x,t) | \le  \omega_{3} (x,t)$\  then\  $$ |\psi (x,t) |\  \le\  C \omega_3^* (x,t).\qquad \quad\qquad\qquad$$

\end{enumerate}
\end{prop}

\proof
Claim \eqref{a4} directly follows from Lemma \ref{lema1}.  We also see that Claims \eqref{a1} and \eqref{a3} follow from Lemma \ref{lema2}. It only remains to prove Claim \eqref{a2}.
Let us write
$$
g_0(x,t)  =        \frac{ \mu_j^{-\frac {n+2}2}} {1+  | \mu_j^{-1} (x-\xi_j ) |^{2+a}}\,  t^{-\sigma} \la_j^{\frac{n-2}2}  \chi\left (  {  \bar\mu_j^{-1}} { |x -\xi_j |}\right)
$$
where $0<a< n-2$. We claim that the following estimate holds: there is a positive supersolution $\psi(x,t)$ of
\be\label{bn}
\psi_t \ \ge\  \Delta \psi  +  g_0(x,t)  \inn  \R^n \times (t_0,\infty)
\ee
such that
$$\begin{aligned}  \psi(x,t)  \ \le &  \  C\, \Big [ \frac{ \mu_j^{-\frac {n-2}2}} {1+  | \mu_j^{-1} (x-\xi_j ) |^{a}}\,  t^{-\sigma} \la_j^{\frac{n-2}2}
 \chi\left (  {  \bar\mu_j^{-1}} { |x-\xi_j |}\right)\\\  &\ \  + \    \frac{ \bar \mu_j^{-\frac {n-2}2}} {1+  | \bar \mu_j^{-1} (x-\xi_j )|^{n-2}}\,  t^{-\sigma} \la_j^{\frac{n-2}4 +\frac a2}   +    t^{-\frac n2 } e^{-\frac {|x-\xi_j |^2}{4t}}   \,   \Big ]
\end{aligned} $$
To prove this, we consider first the problem
$$
\psi^1_t \ \ge\  \Delta \psi^1  + \frac{ \mu_j^{-\frac {n+2}2}} {1+  | \mu_j^{-1} (x-\xi_j )|^{2+a}}\,  t^{-\sigma} \la_j^{\frac{n-2}2}       \inn  \R^n \times (t_0,\infty).
$$
According to Lemma \ref{lema2} there is a positive supersolution of this problem with
$$\begin{aligned}  \psi^1(x,t)  \ \le &  \  C\, \Big [ \frac{ \mu_j^{-\frac {n-2}2}} {1+  | \mu_j^{-1} (x-\xi_j ) |^{a}}\,  t^{-\sigma} \la_j^{\frac{n-2}2}
  +    t^{-\frac n2 } e^{-\frac {|x-\xi_j |^2}{4t}}   \,   \Big ]
\end{aligned} $$
Then, $\psi$ satisfies \equ{bn} if $\psi = \psi^1 \eta   +  \bar\psi $ where
$\eta(x,t) =  \chi ( \bar \mu_j^{-1} (x-\xi_j ) ) $ and
$$
\bar\psi_t  \ge  \Delta \bar \psi   +    2\nn \psi^1\cdot \nn \eta  +  (\Delta\eta  - \eta_t)\psi^1
$$
Now we observe that
$$
|2\nn \psi^1\cdot \nn \eta  +  (\Delta\eta  - \eta_t)\psi^1  | \ \le \  C\, \la_j^{\frac{n-2}4 + \frac a2} t^{-\sigma}  \frac{ \bar\mu_j^{-\frac {n+2}2}} {1+  | \bar \mu_j^{-1} (x-\xi_j ) |^{n+1}}.
$$
The existence of  $\bar \psi$  with the desired bound then follows from Lemma \ref{lema2}. \qed

\medskip
For a function $h=h(x,t)$, we define the norm $\| h \|_{a,\sigma,\beta}$ as the least number $M>0$ such that
$$
|h(x,t) |\ \leq\  M \sum_{j=2}^k\big (\omega_{1j} + \omega_{2j} + \omega_{11}  + \omega_3 \big)(x,t) \foral (x,t)\in \R^n \times (t_0,\infty).
$$
Similarly, we define the the norm $\| h \|_{*,a,\sigma,\beta}$ as the least  $M$ with
\be \label{norm1*}
|h(x,t) |\ \leq\  M \sum_{j=2}^k\big (\omega_{1j}^* + \omega_{2j}^* + \omega_{11}^*  + \omega_3^* \big)(x,t) \foral (x,t)\in \R^n \times (t_0,\infty).
\ee
As a consequence of Proposition \ref{merda} we find the following estimate, fundamental for our purposes.

\begin{corollary}\label{corolnn} There exists a $C>0$ such that for all  $g$ with $\| g \|_{a, \sigma , \beta}<+\infty$ we have
\be \label{apala}
\| \psi \|_{*, a, \sigma, \beta} \leq C \| g \|_{a, \sigma , \beta}.
\ee
where $\psi=\mathcal T^{out}[g]$ is the solution of $\equ{heat}$ given by  $\equ{duh}$.

\end{corollary}

\subsection{The linear inner problems}

Next we will state the necessary facts to solve the inner problems \equ{inner}.
We omit the index $j$ and then consider the linear equation  for $\phi=\phi(y,t)$

\be
\mu(t)^2 \pp_t\phi  = \Delta_y \phi  + pU(y)^{p-1}\phi    + h(y,t)  , \quad t_0 \le  t,\    |y|\le 2R,
\label{linearinner}\ee
where $\mu(t) \sim t^{-1-\sigma}$ for a suitable $\sigma>0$. Here $R$ is a large number, possibly $t$-dependent.

Our purpose is to solve \equ{linearinner} for a $\phi$ that defines a linear operator in $h$ and has good bounds in terms of $h$, provided that
certain solvability conditions for the right hand side are satisfied.

One observation is that the change of variables
$$
\tau(t) =  \tau_0 +   \int_{t_0}^t  \mu(s)^{-2}ds  \sim t^{2\sigma +3}
$$
transforms equation \equ{linearinner} into
\be
\phi_\tau  = \Delta_y \phi  + pU(y)^{p-1}\phi    + h(y,\tau) , \quad \tau_0 \le  \tau,\    |y|\le 2R.  
\label{linearinner1}\ee
Let us recall some basic facts on the elliptic problem for functions $\phi(y)$
\be \label{poto1} L_0[ \phi ] : =\Delta_y \phi   +  pU(y)^{p-1}\phi = h(y) \inn \R^n \ee
Using a decomposition in spherical harmonics
$$ \phi(y) =  \sum_{i=0}^\infty  \phi_i(|y|) \Theta_i(y/|y|) $$
where $\Theta_i$ designates a basis of eigenfunctions of the problem $-\Delta_{S^n} \Theta_i = \mu_i \Theta_i $.
The above system decouples into an infinite set of equations for the radially symmetric coefficients. The following facts are standard.
\begin{enumerate}
\item
The bounded functions satisfying $L_0[Z] = 0$ are precisely the linear combinations of the $n+1$ functions
$$
Z_i(y) =  \pp_{y_i}U (y) , \quad i=1,\ldots, n, \quad   Z_{n+1} (y) =  y\cdot \nn U(y) +  \frac{n-2}2 U(y).
$$

\item
If $ h(y) = O(|y|^{-m})$ as $|y|\to +\infty$, with $2<m<n$, then
Equation \equ{poto1} has a decaying solution $\phi(y) =   O(|y|^{2-m})$ if and only if
$$
\int_{\R^n}  h(y)Z_{i}(y) \, dy  = 0 \foral i=1,\ldots, n, n+1.
$$
In the radial case, this is what formula \equ{formula} directly yields.

\item   The eigenvalue problem
$$
\label{eigen0}
L_0[ f ] = \la f  , \quad  f \in L^\infty(\R^n).
$$
has a unique positive eigenvalue $\la_0>0$, which is simple and with a positive eigenfunction $Z_0(y)$ with
$$Z_0(y) \sim  |y|^{-\frac{n-1}2} e^{-\sqrt{\la_0 }\,  |y|}  \ass |y| \to \infty, $$
which we normalize so that $\int_{\R^n}Z_0^2 dy = 1. $

\end{enumerate}

While $Z_0$ does not enter in solvability conditions in the elliptic problem \equ{poto1}, it plays a crucial role in solving for
$\phi$ uniformly bounded
its parabolic counterpart \equ{linearinner1} in entire space, say
\be \label{poto2} \phi_\tau =   L_0[ \phi ] + h(y,\tau)  \inn \R^n\times (\tau_0 ,\infty) . \ee
In fact if we set
$$p(\tau) = \int_{\R^n}  \phi(y,\tau)Z_0(y)\, dy, \quad   q(\tau) = \int_{\R^n}  h(y,\tau)Z_0(y)\, dy .$$
then we compute
$$
\frac {dp}{d\tau} (\tau) -\la_0 p(\tau)  = q(\tau) .
$$
This ODE has a unique bounded solution
$$
p (\tau)  =      \int_\tau^\infty   e^{\la_0(\tau- s) }\, q(s)\, ds
$$
and hence its initial condition is imposed: we need one linear constraint,
on  the initial value $\phi(y,0)$,
$$
\int_{\R^n}  \phi(y,0)Z_0(y)\, dy\  =\   \langle \ell,  h \rangle :=      \int_{\tau_0}^\infty   e^{ -\la_0 s }\, \int_{\R^n}  h(y,s)Z_0(y)\, dy\, ds\, .
$$


\medskip
Let us consider then the initial value problem for \equ{poto2}
\be \label{poto22}\left\{ \begin{aligned}   \phi_\tau = &  L_0[ \phi ] + h(y,\tau) \quad \inn \R^n\times (\tau_0,\infty)\\ \phi(x,0) = &\ell Z_0(x) \end{aligned} \right. \ee
Let us assume that for some $2<m<n$ and $\nu>0$ the right hand side $h$ satisfy the decay conditions
\be\label{sss}  h(y,\tau) \sim  \frac {\tau^{-\nu}}{1+ |y|^m}, \quad \int_{\R^n} h(y,\tau)\,Z_i(y)\,dy\, =\, 0 , \, i=1, \ldots , n+1, \ee
for all $\tau \in (\tau_0 , \infty)$.
Let us consider as an approximate solution that obtained by solving the elliptic equation
$$
L_0[ \bar \phi ] + h(y,\tau)=0 \quad \inn \R^n
$$
so that \be \label{vpp}\bar \phi(y,\tau ) \sim  \frac {\tau^{-\nu}}{1+ |y|^{m-2}}, \ee
and formally, the error of approximation is given by
$$
 \quad - \bar \phi_\tau (y,\tau ) \sim  \frac {\tau^{-\nu-1}}{1+ |y|^{m-2}}.$$
With this choice  we obtain an improvement in the region  $|y|\ll \sqrt{\tau}$ for error of approximation, since there  $-\bar \phi_\tau$
has smaller size compared with $h$.
We would like to find a true solution of \equ{poto22} with the behavior \equ{vpp}, but according to the above discussion this can only be achieved  with the choice $\alpha = \langle \ell , h  \rangle$. It is then natural to consider the problem restricted to a ball $B_{2R}$ in $\R^n$ as in \equ{linearinner1} where
$R= R(\tau) \ll \sqrt{\tau} .  $  In fact what we can establish is that for $h$ satisfying \equ{sss} there is a solution $\phi$ of \equ{poto22} defined in $B_{2R}$ that satisfies an estimate similar but worse than \equ{vpp}.  That estimate however coincides with \equ{vpp} for
$|y|\sim R $, which is enough for our purposes. Let us be more precise.
Let us denote, for $R(t) = t^\ve$, $\ve>0 $ small and fixed,
$$ \DD_{2R} =  \{ (y,t)  \ /\ t\in (t_0,\infty) , \  |y|\le  2R(t) \} .   $$
 and consider the initial value problem
\begin{equation}\label{linear} \left\{ \begin{aligned}
\mu^2 \phi_t\ &= \  \Delta_y \phi   +  pU(y)^{p-1}\phi   + h(y,t )  \, \inn \DD_{2R}\\
   \phi(y,0)\ &=\ \ell Z_0(y)  \inn B_{2R},
\end{aligned}\right.
\end{equation}
for some constant $\ell $, under
the orthogonality conditions, for $t \in (t_0 , \infty)$
\begin{equation}\label{uffa1}
\int_{B_{2R}} h(y,t)\,  Z_i (y) \, dy = 0  \foral i= 1,\ldots,{n+1}  .
\end{equation}

\medskip
Let us fix numbers $0<a<n-2$ and $\nu>0$ and define the following norms. We let  $\|h\|_{2+a,\nu}$ be the least number $K$ such that
$$
  |h(y,t) |  \ \le    \   K \frac{\mu (t)^{\nu}} { 1+ |y|^{2+a}} \inn \DD_{2R}
\label{norma**}$$
According to the above discussion, in the best of the worlds we would like to find a solution to
\equ{linear} that satisfies $\|\phi\|_{* a,\nu}\le C\|h\|_{2+a,\nu}$.
We cannot quite achieve this but, let us define $\|\phi\|_{* a, \nu}$ to be the least number $K$ with
\be
  |\phi(y,t) |  \ \le    \  K  R^{n+1-a} \frac { \mu (t)^{\nu}} { 1+ |y|^{n+1}} \inn \DD_{2R}.
\label{norma*}\ee
We notice that
$$
|\phi(y,t) |  \le   \|\phi\|_{* a, \nu} \frac{\mu (t)^{\nu}} { 1+ |y|^{a}} \hbox{ for }
$$


\medskip
The following is the key linear result associated to the inner  problem.
\begin{lemma} \label{lema11} There is a $C>0$ such
For all sufficiently large $R>0$ and any  $h$  with   $\|h\|_{2+a,\nu} <+\infty$
that satisfies relations $\equ{uffa1}$
there exist linear operators $$\phi = \TT^{in}_\mu [h],\quad  \ell = \ell [h] $$ which solve Problem $\equ{linear}$ and define linear operators of $h$
with
$$
|\ell [h]|  +    \| (1+|y|)\nn_y \phi \|_{*,a, \nu} \ +\ \|\phi\|_{*,a, \nu} \ \le \ C\,\|h\|_{ \nu, 2+a}.
$$
\end{lemma}

\proof
As we have discussed,  Problem \equ{linear} is equivalent to
$$ \begin{aligned}
 \phi_\tau\ &= \  \Delta_y \phi   +  pU(y)^{p-1}\phi   + h(y,\tau )  \, \inn  |y|\le 2R, \quad \tau\in (\tau_0 ,\infty) \\
   \phi(y,0)\ &=\ \ell Z_0(y)  \inn B_{2R}.
\end{aligned}
$$
The result  then follows from Proposition 5.5 and the gradient estimates in the proof of Proposition 7.2  in \cite{cdm}.
We remark that we also have the validity of a H\"older estimate in space and time with the natural forms.
\qed





\section{Solving the outer and inner problems}\label{s5}

In this section we will solve the outer-inner gluing system \equ{inner}-\equ{outer}, setting it up as a system in $\phi$, $\psi$ and $\vec \mu_1$ that
will involve a fixed point formulation in terms of the linear inverses built in the previous section.

\subsection{The outer problem}
Let us denote, for $R(t) = t^\ve$, $\ve>0 $ small and fixed as in \eqref{epsilon},
$$ \DD_{2R} =  \{ (y,t)  \ /\ t\in (t_0,\infty) , \  |y|\le  2R(t) \}    $$
and  consider a $k$-tuple of $C^1$ functions $$\vec \phi(y,t) = (\phi_1(y,t),\ldots, \phi_k(y,t)), \quad (y,t)\in \DD_R   .  $$

In addition we consider a bounded function $z_*(x) $ that satisfies the assumption \equ{decay} in Corollary \ref{coro1} namely
\be\label{decay1}  |z_*(x)| \ \le \   \frac \delta {1 +|x|^\alpha}   \ee
with $\alpha > \frac {n-2}2$.
Let us consider the solution
$Z^*(x,t)$ of the heat equation
$$ \left \{
\begin{aligned}
Z^*_t  &= \Delta Z^* \inn  \R^n \times (t_0,\infty)\\
Z^*(x,t_0)& = z_*(x).
\end{aligned}  \right.
$$
Then we have
\be \label{estZ*}
\left| Z^* (x,t) \right| \ \le \  {C\over (\sqrt{t} + |x| )^\alpha }.
\ee
Indeed, the solution of the initial value problem is given explicitly by the convolution formula
$$
Z^* (x,t) = {1\over (4\pi t)^{n \over 2} } \int_{\R^n} e^{-{|x-y|^2 \over 4t}} \, z_* (y) \, dy.
$$
It is not restrictive to think that $\alpha <n$. By the decay assumption \eqref{decay1}, for some constant $C$ whose value changes from line to line, we have
$$
\begin{aligned}
|Z^* (x,t) | &\leq { C \over t^{n\over 2}} \int_{\R^n} {e^{-{|x-y|^2 \over 4t}} \over 1+|y|^\alpha }\, dy \leq C \int_{\R^n} {e^{-{|z|^2 \over 4}} \over 1+|x+\sqrt{t} z|^\alpha }\, dz \\
&\leq C {1\over t^{\alpha \over 2}}  \int_{\R^n} {e^{-{|z|^2 \over 4}} \over |{x \over \sqrt{t}}+  z|^\alpha }\, dz \leq C {1\over (t^{1\over 2} + |x| )^\alpha },
\end{aligned}
$$
which concludes the proof of \eqref{estZ*}.

Let us set in the outer problem \equ{outer},
$$
\Psi(x,t)  =  Z^*(x,t) + \psi(x,t) .
$$
and impose initial condition $0$ for $\psi$.
Then the outer problem in terms of $\psi$ becomes
\be \label{outer111}\left\{ \begin{aligned}
\psi_t  =& \Delta_x\psi   +   G( \vec\phi ,  \psi; \vec\mu_1, \vec \xi ) \inn \R^n\times (t_0,\infty)   \\
\psi(\cdot,t_0) =& 0 \inn \R^n,
 \end{aligned}\right. \ee
where  $$G (\vec\phi, \psi;\vec\mu_1 , \vec \xi  )  =   V \psi  +  B[\vec\phi] + VZ^* +    \mathcal N ( \vec\phi , Z^*+ \psi; \vec \mu_0+ \vec \mu_1, \vec \xi)  +   E^{out} $$
and the components of $G$ are defined in \equ{poto}. We express \equ{outer111} as
\be\label{out}
\psi\  =\  \TT^{out}[\, G(\vec \phi,\psi; \vec\mu_1, \vec \xi)\, ] .
\ee
where $\psi=\TT^{out}[g]$ is the solution of the heat equation given by \equ{duh}.

\subsection{The inner problem}
We will formulate the inner problem \equ{inner} for the functions $\phi_j(y,t)$  using the setting introduced in Lemma \ref{lema11}. 
Let us write Problem \equ{inner} in the form
 \be \label{inner12}
  \mu_j^2 \pp_t \phi_j  =  \Delta_y \phi_j +  pU(y)^{p-1}  \phi_j  +  H_j (\psi,\vec\mu_1, \vec \xi) \inn \DD_{2R}
\ee
where
$$
H_j (\psi,\vec\mu_1, \vec \xi ) = \tilde H_j (\psi,\vec\mu_1, \vec \xi )  +    D_j[\vec\mu_1, \vec \xi]
$$
with
\be \label{Htilde}
\tilde H_j (\psi,\vec\mu_1, \vec \xi ) = \zeta_j U(y)^{p-1}  \mu_j(t)^{\frac{n-2}2} (Z^* + \psi )
\ee
and, as in \equ{e55}-\eqref{D0}  
$$
\begin{aligned}
 D_j[\mu_1 , \vec \xi ](y,t) &=  (\dot \mu_{0j} \mu_{1j}(t) + \mu_{0j} \dot \mu_{1j}(t) )  Z (y)  +   \frac{n-2}2 p U (y)^{p-1} U(0) \la_{0j}(t)^{n-4 \over 2}   {\mu_{1j} \over \mu_{0,j-1} } (t)\\
 &+ \mu_j \dot \xi_j \cdot \nabla U (y) , \quad {\mbox {for}} \quad j=2, \ldots , k\\
 D_1[\mu_1 , \vec \xi ](y,t) &=  \dot (1+ \mu_{11} ) [ \dot \mu_{11} Z(y) + \dot \xi_j \cdot \nabla U(y) ] .
 \end{aligned}
 $$
First we modify the right hand side of \equ{inner12}  to achieve the solvability conditions \equ{uffa1}, and introducing an initial condition as in \equ{poto22}. We consider the problem
\be\label{inner21} \left\{
\begin{aligned}
\mu_j^2 \pp_t \phi_j  =  &\Delta_y \phi_j +  pU(y)^{p-1}  \phi_j  + H_j(\psi,\vec\mu_1 , \vec \xi) - \sum_{i=1}^{n+1} d_{ji}[\psi,\vec\mu_1, \vec \xi]\, Z_i  \inn \DD_{2R}\\
\phi^1(\cdot, t_0) = &\ell Z_0(y) .
\end{aligned} \right.
\ee
where
$$\begin{aligned}
d_{ji}[\psi,\vec\mu_1 , \vec \xi ](t)\ = &\  \frac{\int_{B_{2R}}  H_j(\psi,\mu_1, \vec \xi )(y,t) Z_i (y)\, dy }{ \int_{B_{2R}}   Z_i (y)^2\, dy}.
\end{aligned}
$$
Let us denote by $ \TT^{in}_{\mu_j } $ the linear operator in Lemma \ref{lema11} for $\mu = \mu_j$. Then \equ{inner21} is solved if the following equation holds
\be\label{inj}
\phi_j\  =\  \TT^{in}_{\mu_j}[\, H_j(\psi,\vec\mu_1, \vec \xi )\, ] , \quad j=1, \ldots , k.
\ee
In order to solve the full equation \equ{inner}, we couple these equations with
\be\label{d22}
 d_{ji} [\psi,\vec \mu_1, \vec \xi ] (t) = 0 \foral t\in (t_0,\infty) , \quad j=1, \ldots , k, \, i=1, \ldots , n, n+1.
\ee
The equations \eqref{d22} can be expressed in a quite simple form: Equations \eqref{d22} for $i=n+1$,  $d_{j, n+1} [ \psi , \vec \mu_1 ,\vec \xi  ] = 0 $ are equivalent to
\be
\dot \mu_{1j} + {n-4 \over 2} {\alpha_j \over t}\mu_{1j} +  M_{j, n+1} [\psi , \vec \mu_1 , \vec \xi ]  =0 , \quad t \in (t_0 , \infty ),
\label{bbb}\ee
where $M_{j, n+1} = M_{j, n+1} [\psi , \vec \mu_1 , \vec \xi ]$ is given by  $$
M_{j, n+1} [\psi , \vec \mu_1, \vec \xi ] (t) = {\theta(t)\over t R^2}  \mu_{1j}(t)+
  {\mu_j^{n-2 \over 2}  \over \mu_{0j} }(t){\int_{B_R} p U^{p-1} (y) Z_{n+1} (y) ( \psi + Z_* ) (\mu_j y +\xi_j  , t) \, dy \over \int_{B_R} Z_{n+1}^2 (y) \, dy},
$$
and $\theta(t)$ is a bounded function. Furthermore, Equations \eqref{d22} for $i=1, \ldots , n$,  $d_{j, i} [ \psi , \vec \mu_1 ,\vec \xi  ] = 0 $ are equivalent to
\be \label{bbbp}
\dot \xi_j + M_j [\psi , \vec \mu_1 , \vec \xi] =0 ,
\ee
where
$$
M_{j} [\psi , \vec \mu_1 , \vec \xi] = \mu_j^{n-4 \over 2} (t){\int_{B_R} p U^{p-1} (y) \nabla U(y) ( \psi + Z_* ) (\mu_j y +\xi_j  , t) \, dy \over \int_{B_R} ({\partial U \over \partial y_j} )^2 (y) \, dy} .
$$
\begin{proof}[Proof of \eqref{bbb}-\eqref{bbbp}] Formula \eqref{bbbp} follows by a straightforward computation. Let us consider \eqref{bbb}.
For each $j$, we have
\be\label{gg1}
\begin{aligned}
d_{j, n+1} [\psi , \vec \mu_1 , \vec \xi  ](t)\ =&\quad
  (\dot \mu_{0j} \mu_{1j}(t) + \mu_{0j} \dot \mu_{1j}(t)   )\\ & \quad - {n-2 \over 2} { U(0) \int_{B_R} p U^{p-1} (y) Z_{n+1} (y) \, dy \over
\int_{B_R} Z_{n+1}^2 (y) \, dy} \la_{0j}^{n-4 \over 2} {\mu_{1j} \over \mu_{0, j-1}}(t) \\
&\quad +
{  \mu_j^{n-2 \over 2} \int_{B_R} p U^{p-1} (y) Z_{n+1} (y) (\psi + Z^* ) (\mu_j y +\xi_j  , t) \, dy \over \int_{B_R} Z_{n+1}^2 (y) \, dy}
\end{aligned}
\ee
 Since
$$
\begin{aligned}
\int_{B_R} p U^{p-1} (y) Z_{n+1} (y) \, dy &=
\int_{\R^n } p U^{p-1} (y) Z_{n+1} (y) \, dy + O({1\over R^2} )\\
\int_{B_R} Z_{n+1}^2 (y) \, dy &= \int_{\R^n } Z_{n+1}^2 (y) \, dy + O({1\over R^{n-4}} ),
\end{aligned}
$$
we get
$$
- { U(0) p \int_{B_R}  U^{-1} (y) Z_{n+1} (y) \, dy \over
\int_{B_R} Z_{n+1}^2 (y) \, dy} = c + O({1\over R^2} )
$$
where $c$ is the positive constant defined in \eqref{defc}.
We conclude that the first two terms in \eqref{gg1} are given by
$$
\begin{aligned}
&(\dot \mu_{0j} \mu_{1j}(t) + \mu_{0j} \dot \mu_{1j}(t) ) - {n-2 \over 2} { U(0) \int_{B_R} p U^{p-1} (y) Z_{n+1} (y) \, dy \over
\int_{B_R} Z_{n+1}^2 (y) \, dy} \la_j^{n-4 \over 2} {\mu_{1j} \over \mu_{0, j-1}}\\
&\quad  =
\mu_{0j} \dot \mu_{1j} + {\mu_{1j} \over \mu_{0j}} \left[ \mu_{0j} \dot \mu_{0j} + {n-2 \over 2} (c+ O({1\over R^2} ) ) \la_{0j}^{n-2 \over 2} \right]
\\
& \quad = \mu_{0j} \left[ \dot \mu_{1j} + {n-4 \over 2} (c + O({1\over R^2} )) { \la_{0j}^{n-2 \over 2} \over \mu_{0j}^2} \mu_{1j} \right] \\
& \quad =  \mu_{0j} \left[ \dot \mu_{1j} + {n-4 \over 2}  { \alpha_j   \over t}  \mu_{1j} +    {  O({1\over R^2} )  \over t}  \mu_{1j}  \right].
\end{aligned}
$$
\end{proof}
Next we formulate Equation  \equ{bbb} as a fixed point problem using the initial value problem
\be\label{poto5}\left \{
\begin{aligned}
\dot \mu + {n-4 \over 2} {\alpha_j \over t}\mu\  = &\ \beta (t), \quad t \in (t_0 , \infty )\\
\mu_{1j}(t_0) \ = &\  0
\end{aligned}\right. \ee
We recall that $\mu_{0j} (t) \sim t^{-\alpha_j } $ and we want to solve this equation for a function $\mu(t)$ with a decay slightly faster than this. For a number $b >0$ and a function $g(t)$ we define
\be\label{normg}
\|g\|_{b} :=  \sup_{t>t_0} | t^{b} g(t)|
\ee
The unique solution $\mu(t)$ of  \equ{poto5} defines a linear operator of $\beta(t)$ represented as
$$
\mu(t) = {\mathcal S}_j [\beta](t) := t^{-{n-4 \over 2} \alpha_j} \int_{t_0}^t s^{{n-4 \over 2} \alpha_j} \beta (s) \, ds
$$
Clearly, if $b < {n-4 \over 2} \alpha_j $ for any $j=2, \ldots , k$, we have the validity of the uniform $C^1$-estimate
\be
\|\dot\mu\|_{ b +1}   + \|\mu\|_{b }   \le C \|\beta\|_{b +1} ,\quad      \mu(t) = {\mathcal S}_j [\beta](t)
\label{boundS}\ee
Also, we write Equations \eqref{bbbp} as a fixed point problem using the solution to
$$
\dot \xi = \Xi (t), \quad t \in (t_0 , \infty)
$$
defined by
$$
\xi(t) = {\mathcal P} (\Xi ) := \int_{t}^\infty \Xi (s) \, ds,
$$
for vector-valued function $\Xi$. We have the validity of the uniform $C^1$-estimate
$$
\|\dot\xi\|_{ b +1}   + \|\xi\|_{b }   \le C \|\Xi\|_{b +1} ,\quad      \xi(t) = {\mathcal P} [\Xi](t)
$$

We formulate  equations \eqref{d22} as follows
$$
\left\{ \begin{aligned}
 \mu_{1j} & = {\mathcal S}_j  [-M_{j, n+1} (\psi , \vec \mu_1 , \vec \xi ) ] \\
 \xi_j &= {\mathcal P} [-M_{j} (\psi , \vec \mu_1 , \vec \xi ) ] .
 \end{aligned}
 \right.
$$

\subsection{The system}
Solving the system of $(k+1)$ equations given in  \eqref{inner}-\eqref{outer} reduces to solving \eqref{out}-\eqref{inj}-\eqref{d22} in $\psi$, $\vec \phi$, $\vec \mu_1$ and $\vec \xi $. We formulate the system \eqref{out}-\eqref{inj}-\eqref{d22} as follows
\be\label{sistema}
\begin{aligned}
\psi  \ =  &\ \TT^{out} [ G( \vec \phi, \psi, \vec \mu_1 , \vec \xi ) ]\\
\phi_j\ =  &\  \TT^{in}_{\mu_j}  [ \ttt H_j ( \psi , \vec \mu_1 , \vec \xi ) ] , \quad j=1,\ldots, k \\
\mu_{1j}\  = &\  \mathcal S_j [ M_{j, n+1} (\psi, \vec \mu_1 , \vec \xi)] , \,  \quad j=1,\ldots, k \\
\xi_j \ =&\ \mathcal P [ M_j (\psi, \vec \mu_1 , \vec \xi)] , \,  \quad j=1,\ldots, k.
\end{aligned}
\ee
We now fix the number $\sigma $ introduced in \eqref{assmu} to be some small, positive number satisfying $\sigma < {n-6 \over 2} \alpha_j$ for all $j=2, \ldots , k$. We assume in what follows that
\be\label{cotmu}
\|\vec \mu_1\|_\sigma :=    \|\dot \mu_{11} \|_{1+\sigma }+  \sum_{j=2}^k \|\dot \mu_{1j} \|_{1+ \alpha_j +\sigma }  +  \|\mu_{1j}\|_{\alpha_j +\sigma} \le 1
\ee
and
\be \label{cotpoint}
\|\vec \xi\|_{\diamondsuit,\sigma} :=     \sum_{j=1}^k \|\dot \xi_{j} \|_{1+ \alpha_j +\sigma }  +  \|\xi_{j}\|_{\alpha_j +\sigma} \le 1
\ee
where the norm $\| \ \|_b$ is defined in \equ{normg}.
The function $\psi$ will naturally be measured in the norm  \equ{norm1*}.
Let us assume that $\|\psi\|_{*,\sigma',a,\beta'} <+\infty$ for some number $\sigma'> \sigma$, $0<a<1$  and $\beta' > 2+\alpha$
where the number $\alpha>\frac {n-2}2$ chosen in the bound
\equ{estZ*} for $Z^*$. Here  $\sigma$ is already fixed in bound \equ{cotmu}.
Then the following pointwise estimate for the operator $\tilde H_j$ in \eqref{Htilde} holds:
\be \label{estihj}\begin{aligned}
|\ttt H_j(\psi, \vec \mu_1 , \vec \xi ) (y,t)|  \ \le & \ \frac C{1+|y|^4} \big [ \,
\| \mu_j(t)^{\frac{n-2}2}  (\psi+ Z_*) (\cdot, t) \|_{L^\infty( R^{-1}\mu_j <|x-\xi_j|< R\mu_j)}  \big]\\ \le&\
\frac C{1+|y|^4} \, {\la_j^{\frac{n-2}2} t^{-\sigma'}}\big [ \,
\|\psi\|_{*,\sigma',a,\beta'} +  1 \big ].
\end{aligned} \ee
Consistently, for the operator $M_{j, n+1}$ we find, using that $|\mu_{0j}^{-1} \la_j^{\frac{n-2}2}|\le Ct^{-\alpha_j-1}$,
\be\label{hj1}
|M_{j, n+1} (\psi, \vec \mu_1 , \vec \xi )(t)|  \ \le\  \, {   t^{-\alpha_j -1 -\sigma'}}\big [ \,
\|\psi\|_{*,\sigma',a,\beta'} +  1 \big ].
\ee
Moreover, consequence of Lemma \ref{lema2}, we have
\be \label{estihjg}\begin{aligned}
|\ttt H_j(\psi, \vec \mu_1 , \vec \xi ) (y,t)& - \ttt H_j(\psi, \vec 0 , \vec 0 ) (y,t)|  \\
\ \le & \ \frac C{1+|y|^4} \big [ \,
\| \mu_j(t)^{\frac{n}2}  \nabla_x (\psi+ Z_*) (\cdot, t) \|_{L^\infty( R^{-1}\mu_j <|x-\xi_j|< R\mu_j)}  \big]\\ \le&\
\frac C{1+|y|^4} \, {\la_j^{\frac{n-2}2} t^{-\sigma'} \over R }\big [ \,
\|\psi\|_{*,\sigma',a,\beta'} +  1 \big ].
\end{aligned} \ee
Consistently, for the operator $M_{j}$ we find, using that $|\mu_{0j}^{-1} \la_j^{\frac{n-2}2}|\le Ct^{-\alpha_j-1}$,
\be\label{hj1g}
|M_{j} (\psi, \vec \mu_1 , \vec \xi )(t)|  \ \le\  \, {   t^{-\alpha_j -1 -\sigma'} \over R }\big [ \,
\|\psi\|_{*,\sigma',a,\beta'} +  1 \big ].
\ee
Bounds \equ{estihj}-\eqref{estihjg} and \equ{hj1}-\eqref{hj1g} roughly  tell us  that in the inner problems
$\phi_j = \TT^{in}_{\mu_j} [H_j ] $, the norm
$\|\phi_j\|_{*,a,\nu_j}$  in \equ{norma*} is expected to be bounded, where $0<a \le 2$, $\mu= \mu_j$ and $\nu_j$ is the power such that $\mu_j(t)^{\nu_j}  \sim t^{-\alpha_j-1-\sigma  }$. Let us also fix $a=1$. We write in what follows
\be\label{cotaphi}
 \|\phi_j\|_{j,\sigma} := \|\phi_j\|_{*,1,\nu_j}, \quad   \|\vec\phi\|_{\sigma} :=  \sum_{j=1}^k \|\phi_j \|_{j,\sigma }  .
\ee
Let us choose numbers
\be \label{fixthem}
0<a<1,  \quad 0<\sigma< \sigma' < \sigma'' <1, \quad     2+ \alpha <\beta'  <  \alpha p  \quad \beta''= p\alpha .
\ee
and measure $\psi$ in the norm  $\|\psi\|_{*,a, \sigma' ,\beta'} $.
A major role in the rest of the proof  will be played  by the following estimate for the operator $G$.

\begin{prop}\label{PestG} Assume the parameters $\mu_j $ have the form \eqref{muj0} with $\vec \mu_1$ satisfying $\equ{cotmu}$, and the points $\vec \xi$ satisfy \eqref{cotpoint}. Then there exists   $\ell >0$ such that for all sufficiently large $t_0$ we have
\be\label{estG}
\| G (\vec\phi, \psi;\vec\mu_1 ,\vec \xi)  \|_{a,\sigma'', \beta''} \ \leq\   t_0^{-\ell} \big( 1+ \| \psi \|_{*,a, \sigma' , \beta'} +
\| \psi \|^p_{*,a, \sigma' , \beta'}+  \| \vec \phi \|_\sigma
 + \| \vec \phi \|_\sigma ^p  \big).
\ee
\end{prop}

\bigskip
\subsection*{The fixed point problem }
First we define the space $X$ of tuples of functions  $(\vec \phi, \psi, \vec \mu_1, \vec \xi) $  where $\vec \phi$, $\nn_y \phi$ are continuous in $\bar \DD_{2R}$, $\psi $ is continuous in $\R^n\times [0,\infty)$ and $\vec \mu_1$, $\vec \xi$ is of class $C^1[t_0,\infty)$ and such that
$$  \|(\vec \phi, \psi, \vec \mu_1, \vec \xi) \|_X:=    \|\vec \mu_1\|_\sigma + \| \vec \xi \|_{\diamondsuit , \sigma}+ \|\vec \phi\|_\sigma  + \|\psi\|_{*,\sigma', \beta', a}  < +\infty. $$
$X$ is a Banach space endowed with this norm. It is convenient to formulate the fixed point problem \equ{sistema} in the form
\be\label{fp2} (\vec \phi, \psi, \vec \mu_1, \vec \xi)  =
 \vec T(\vec \phi, \psi, \vec \mu_1, \vec \xi)  \inn X
\ee
where $\vec T = (\vec T^1,T^2,\vec T^3, \vec T^4)$, with
$$\begin{aligned}
T^2 ( \vec \phi, \psi, \vec \mu_1, \vec \xi) \ =& \ \TT^{out} [ G( \vec \phi, \psi, \vec \mu_1, \vec \xi) ]\\
 \vec T^{1}_j (\vec \phi, \psi, \vec \mu_1, \vec \xi)\ = &\ \TT^{in}_{\mu_j}  [ H_j ( \psi , \vec \mu_1 , \vec \xi) ] , \\
  \vec T^{3}_j (\vec \phi, \psi, \vec \mu_1, \vec \xi )\ = &\  \mathcal S_j [ M_{j, n+1} ( \TT^{out}[ G( \phi,\psi,\vec\mu    , \vec \xi), \vec \mu_1 , \vec \xi )],\\
    \vec T^{4}_j (\vec \phi, \psi, \vec \mu_1, \vec \xi )\ = &\  \mathcal T [ M_j ( \TT^{out}[ G( \phi,\psi,\vec\mu  , \vec \xi), \vec \mu_1 , \vec \xi)]
  \quad j=1,\ldots, k .
\end{aligned} $$
Using Estimates \equ{apala} and \equ{estG}
we get that
\be \label{estii}\begin{aligned}
\| T^2 ( \vec \phi, \psi, \vec \mu_1 , \vec \xi ) \|_{*,\sigma'', \beta'\, a} \le & \  C \|G( \vec \phi, \psi, \vec \mu_1 , \vec \xi) \|_{\sigma''\, \beta'', a}\qquad \qquad \qquad \qquad\qquad \qquad \\ \le & \
  t_0^{-\ell} \big( 1+ \| \psi \|_{*,a, \sigma' , \beta'} +
\| \psi \|^p_{*,a, \sigma' , \beta'}+  \| \vec \phi \|_\sigma
 + \| \vec \phi \|_\sigma ^p  \big).
\end{aligned} \ee
for some $\ell >0$.
Now, from estimate
\equ{estihj} and the bounds in the definition of the operator
$ \TT^{in}_{\mu_j}$ we see that
$$
     \| T^1_j ( \vec \phi, \psi, \vec \mu_1 , \vec \xi) \|_{j,\sigma'} \ \le\      C\big( 1+ \| \psi \|_{*,a, \sigma' , \beta'})
$$
and hence for some $\ell>0$,
\be
 \| T^1_j ( \vec \phi, \psi, \vec \mu_1 , \vec \xi) \|_{j,\sigma}\ \le\  t_0^{-\ell} \big( 1+ \| \psi \|_{*,a, \sigma' , \beta'}).
\label{t12}\ee
Similarly, estimates \equ{boundS}, \equ{hj1} and  \equ{estii} yield

\be\label{t3}
\begin{aligned}
 \|T^3_j( \vec \phi, \psi, \vec \mu_1 , \vec \xi) \|_{\sigma}  \ \le&\    t_0^{-\ell} \|T^3_j( \vec \phi, \psi, \vec \mu_1) \|_{\sigma'} \\
 \le  &\
 C t_0^{-2\ell} \big( 1+ \| \psi \|_{*,a, \sigma' , \beta'} +
\| \psi \|^p_{*,a, \sigma' , \beta'}+  \| \vec \phi \|_\sigma
 + \| \vec \phi \|_\sigma ^p  \big).
\end{aligned} \ee
Let $$B= \{ ( \vec \phi, \psi, \vec \mu_1 , \vec \xi)\in X \ /\ \| ( \vec \phi, \psi, \vec \mu_1 , \vec \xi) \|_X \le 1\}. $$
From estimates \equ{estii}, \equ{t12} and \equ{t3}  we  find
$$  \|  \vec \phi, \psi, \vec \mu_1 , \vec \xi) \|_X  \le  K t_0^{-\ell} \foral ( \vec \phi, \psi, \vec \mu_1 , \vec \xi) \in B . $$
for some fixed $K$. Hence, enlarging $t_0$ if necessary, we find  that
$$
\vec T ( B) \subset B .
$$
The existence of a fixed point in $B$ will then follow from Schauder's theorem  if we establish the compactness of the operator in the
topology of $X$.

\subsection*{Compactness}

Thus we consider a sequence of parameters
$(\vec\phi_n,\psi_n, \vec\mu_{1n})\in X$  which is bounded.
We have to prove that the sequence $\vec T (\vec\phi_n,\psi_n, \vec\mu_{1n} , \vec \xi_n )$ has  a convergent subsequence in $X$.
Let us consider first the sequence

$$
\bar\phi_{nj} :=  T^1_j  (\vec\phi_n,\psi_n, \vec\mu_{1n} ,  \vec \xi_n ) =    \TT^{in}_{\mu_j}  [ h_n],  \quad h_n:=  \ttt H_j ( \psi_n , \vec \mu_{in} , \vec \xi_n).
$$
From the above estimates, we see that $h_n$ is locally  uniformly bounded in $\bar \DD_R$, and then so is $\bar\phi_{nj}(y,t)$.
Writing $\bar\phi_{nj}(y,t) =  \ttt \phi_{nj}(y,\tau_n(t)) $ where $\tau_n(t) =  \int_{t_0}^t \frac {ds}{\mu_{jn}(s)^2} $, we see that $\ttt \phi_{nj}(y,\tau)$ satisfies
$$
\pp_\tau \ttt \phi_{nj} = \Delta_y \phi_{nj} + \ttt h_n (y,\tau)
$$
where $\ttt h_n (y,\tau) $ is uniformly bounded, and with a uniformly smooth initial condition.
We conclude that for any compact set $K\subset \bar \DD_{2R}$ and points $(x_l,\tau_l)\in \DD_{2R}$
we have that for a fixed $\gamma\in (0,1)$
$$
|\nn_y \ttt\phi_{nj} (y_1,\tau_1) - \nn_y \ttt\phi_{nj} (y_2,\tau_2) | \le C_K [\, |y_1-y_2|^\gamma +  |\tau_1- \tau_2|^{\frac \gamma 2} ]
$$
But if $\tau_l= \tau_n(t^l)$
$$
|\tau_n (t_1) - \tau_n(t_2)|  \le C_K  \max_{s\in [0,T]} {\mu_n(s)^{-2}}    |t_1-t_2|
$$
for a certain fixed $T>0$. This estimates implies the local equicontinuity of the sequences $\bar \phi_{nj}(y,t)$ and $\nn_y \bar \phi_{nj}(y,t)$.
Hence there is a subsequence (that we still denote the same way) such that they converge uniformly on compact subsets of $\DD_{2R}$.
Now, from the assumed bound on $\psi$ and the a priori estimate obtained we have that
$$
|\bar \phi_{nj}(y,t)|  + (1+|y|) |\nn \bar \phi_{nj}(y,t)| \   \le\         Ct^{-\sigma'}\la_j^{\frac{n-2}2}    \frac{ R^n}{1+ |y|^{n+1}}
$$
which implies that for some $\gamma>0$ and some $\sigma <\sigma'' < \sigma'$,
$$
|\bar \phi_{nj}(y,t)|  + (1+|y|) |\nn \bar \phi_{nj}(y,t)| \   \le\         Ct^{-\sigma''}\la_j^{\frac{n-2}2}    \frac{ R^n}{1+ |y|^{n+1+\gamma}}.
$$
This implies that the local uniform limit of $\bar \phi_{nj}(y,t)$ is in fact global in the norm $\| \  \|_{j,\sigma}$. This gives the
compactness.

\medskip
Let us consider now
$$
\bar \psi_n = T^2 (\vec\phi_n,\psi_n, \vec\mu_{1n} , \vec \xi_n ) =\TT^{out} [ g_n],  \quad g_n = G(\phi_n,\psi_n, \vec\mu_{1n}, \vec \xi_n )
$$
Since $g_n$ is uniformly bounded, we have a uniform H\"older bound for $\bar \psi_n(x,t)$ on compact subsets of $\R^n\times [0,\infty)$.
Hence $\bar \psi_n(x,t)$ converges uniformly (up to a subsequence) to a function $\bar\psi$.  The convergence also holds in the norm
$
\| \  \|_{*, \sigma' , a, \beta'}
$
as it follows of the further uniform decay (in space and time) quantitatively measured by the boundedness of the operator $G$  in the norm $\| \  \|_ {\sigma'' , a, \beta''}$ This convergence yields in straightforward way that of $T^3 (\vec\phi_n,\psi_n, \vec\mu_{1n} , \vec \xi_n)$. The proof is
concluded.

\subsection*{Conclusion} From Schauder's Theorem, we have the existence of a solution to the fixed point $(\vec \phi,\psi, \vec \mu_1 , \vec \xi)$
problem \equ{fp2} in $B$. In fact we
see that  $\|(\vec \phi,\psi, \vec \mu_1 , \vec \xi )\|_X \le  \delta$ for any given small $\delta>0$ after having chosen $t_0$ sufficiently large.
Then we have that the function

$$
u[\vec \mu , \vec \xi ] (x,t) := u_*[\vec\mu , \vec \xi ]   + \psi(x,t)  + \sum_{j=1}^k \eta_j \frac 1{\mu_j^{\frac {n-2}2}} \phi_j \left ( \frac {x-\xi_j }{\mu_j},t \right ) + Z^*(x,t)
$$
is a solution to \eqref{F1}, with energy density satisfying \eqref{ed}. Notice that by construction $\vec \mu (t_0 ) = \vec \mu_0 (t_0)$ since $\vec \mu_1 (t_0)=0 $,  $\psi(x,t_0)=0$   while
$$
 \phi_j \left ( \frac {x -\xi_j } \mu  , t_0 \right ) = \ell_j Z_0  \left ( \frac {x-\xi_j (t_0)}{\mu_j^0(0)} \right ),
$$
with $\lim_{t \to \infty} [ \xi_{j+1}  (t) - \xi_j (t) ] = 0$ for $j=1, \ldots , k-1$.
Write $\ell_j =\ell_j^0+\ell_j^1  $ and $\vec \xi = \vec \xi^0 + \vec \xi^1$  where $\ell^0$  and $\vec \xi^0$ are the adequate values for $z_*\equiv 0$ and $u_0(x)$ the corresponding initial condition.  To have the energy density satisfying \eqref{ed}, $\xi^0_1=0$. Then by construction we find that $\ell^1 $ and $\vec \xi^1 (t_0)$ have a size proportional to that of $z_*(x)$. Since   We have then that the solution of equation \equ{F1} with initial condition
$$
u(x,t_0) = u_0(x) + z_*(x)  +  \sum_{j=1}^k \ell_j^1 \omega_j(x) +\sum_{j=2}^k \xi^1_j (t_0) \, \cdot \, \nabla \tilde \omega_j (x)
$$
gives rise to a $k$-bubble tower if we choose the compactly supported functions
$$
\omega_j(x) =  Z_0  \left ( \frac {x-\xi_j (t_0) } {\mu_j^0(t_0)} \right ) \chi\left( \frac {x-\xi_j (t_0 )} {R\mu_j^0(t_0)}\right),
$$
$$
\tilde \omega_j(x) =  U  \left ( \frac {x-\xi_j (t_0) } {\mu_j^0(t_0)} \right ) \chi\left( \frac {x-\xi_j (t_0 )} {R\mu_j^0(t_0)}\right).
$$
The proof is concluded. \qed

\bigskip
\begin{remark} \label{rem}{\em
The set of initial conditions of Equation \equ{F} that lead to a $k$-bubble tower can actually be smoothly parametrized.
In fact we have used Schauder's theorem just for simplicity. With a bit more effort and similar computations we can prove that the operator
$\vec T$ in the fixed problem \equ{fp2} is actually a contraction mapping, and then we have uniqueness of the fixed point and
the scalars $\ell_j^1 = \ell_j^1 [z_*] $, $\vec \xi^1 [z_*]$ defines a Lipschitz function of $z_*$ in its natural topology. Moreover, an application of
implicit function theorem, yields the $C^1$character of these  functionals. It follows that this set of initial conditions defines a codimension $k+ n (k-1)$ manifold inside the natural finite-energy space of perturbations.
}

\end{remark}

\section{Estimates for the operator $G$}

In this section we will prove  Proposition \ref{estG}.  The result will follow from individual estimates for each of the terms of the operator
which we state and prove as separate lemmas below. We recall that
$$G (\vec\phi, \psi;\vec\mu_1 , \vec \xi )  =   V \psi  +  B[\vec\phi] + VZ^* +    \mathcal N ( \vec\phi , Z^*+ \psi; \vec \mu_0+ \vec \mu_1 , \vec \xi)  +   E^{out} . $$
All terms are defined in \eqref{poto}, with $\Psi = \psi + Z^*$.

\medskip
\begin{lemma} \label{lemma1} Assume the parameters $\mu_j $ have the form \eqref{muj0} with $\vec \mu_1$ satisfying $\equ{cotmu}$, and the points $\vec \xi$ satisfy \eqref{cotpoint}. Assume that $a$, $\sigma''$ and $\beta'' $ are defined as in \eqref{fixthem}.  Then there exists $\ell >0$ so that for all sufficiently large $t_0$ we have
\be \label{esterrore}
\| E^{out}  \|_{ a , \sigma'' , \beta'' } \, \le\, \,t_0^{-\ell},
\ee
where $E^{out} $ is defined in \eqref{poto}.
\end{lemma}

\begin{proof} The error function $E^{out}$ defined in \eqref{poto}
has the explicit expression
\begin{equation}\label{e62}
\begin{aligned}
E^{out} &= \bar E_{11} + \bar E_2 +   \sum_{j=2}^k  {(-1)^{j-1} \over \mu_j^{n+ 2 \over 2}} \theta_j (\vec \mu_1 , \vec \xi ) \chi_j + {\bar \chi \over \mu_1^{n+2 \over 2} } \theta_1 (\vec\mu_1 , \vec \xi ) \\ &
+ \sum_{j=2}^k p (f'(\bar U)- f'(U_j) ) \varphi_{0j} \chi_j     + \sum_{j=2}^k \left[ 2 \nabla_x \varphi_{0j} \nabla_x (\chi_j  ) + \Delta_x (\chi_j  ) \varphi_{0j} \right] \\
  &
  - \sum_{j=2}^k \pp_t (\varphi_{0j} \chi_j  )  +  N_{\bar U} [\varphi_0 ] + \mu_{1}^{-{n+2 \over 2}} D_1 [\vec \mu_1 , \vec \xi]  [\bar \chi - \eta_1]
  \\
  & + \sum_{j=2}^k (-i)^{j+1} \mu_j^{-{n+2 \over 2}} D_j [\vec \mu_1 , \vec \xi] [\chi_j - \eta_j].
\end{aligned}
\end{equation}
In the above formula, $\bar E_{11}$ is defined in \eqref{e11n}, $\bar E_2$ in \eqref{e1}, $D_1 [\vec \mu_1 , \vec \xi]$ in \eqref{D0}, for $j=2, \ldots , k$, $\theta_j $ and $D_j [\vec \mu_1 , \vec \xi]$ in \eqref{e55},
$\chi_j$ in \eqref{defchij}  and $\varphi_{0j} $ in \eqref{defvarphi0}-\eqref{defphi0j}. Also the functions $\eta_j$, $j=1, \ldots , k$ are the cut-off functions introduced in \eqref{newcuts}.

We will estimate in details most of the terms in \eqref{e62}. To estimate these terms in the norm $\| \cdot \|_{ a , \sigma'' , \beta'' }$, we will make use od the definition of the weights $\omega_{11}$, $\omega_{1j}$, $\omega_{2j}$, $j=2, \ldots , k$, and $\omega_3$ introduced respectively in
\eqref{omega1}, \eqref{omega1j}, \eqref{omega2j} and \eqref{omega3}, with $\sigma = \sigma''$ and $\beta = \beta''$.

\medskip
\subsubsection*{Estimate of $\bar E_{11}$  in \eqref{e11n}}
 We start with the term $$\sum_{j=2}^k f' (U_j) \left( \sum_{i \not= j, j-1} U_i \right) \chi_j.$$
  Let us fix $j$. If $i \leq j-2$, we have
$$
\begin{aligned}
\left| f' (U_j ) U_i \chi_j \right| &\leq C \la_{j-1}^{n-2 \over 2} \,  {\la_j^{n-2 \over 2} \over \mu_j^{n-2 \over 2}}  {\mu_j^{-2} \over 1+ |{x -\xi_j  \over \mu_j} |^4} \leq t_0^{-\ell}  \omega_{1j} (x,t),
\end{aligned}
$$
for some $\ell >0$.
If $i >j$,
$$
\begin{aligned}
\big| & f' (U_j ) U_i \chi_j \big| \ \leq  \  C {1\over \mu_j^2} {1\over 1+ |{x-\xi_j \over \mu_j} |^4} U_i (x) \chi_j (x) \\
\leq &\
C {1\over \mu_j^2} {1\over 1+ |{x-\xi_j \over \mu_j} |^4} {\mu_{j+1}^{n-2 \over 2} \over (\mu_{j+1}^2 + |x-\xi_{j+1}|^2)^{n-2 \over 2} }  \chi_j (x) \left( {\bf 1}_{\{ |x-\xi_j | < \mu_j^{1+a} \} } + {\bf 1}_{\{ |x-\xi_j | > \mu_j^{1+a} \} } \right),
\end{aligned}
$$
for some $a>0$, where ${\bf 1}$ is the function defined in \eqref{funo}.
Now, we observe that
$$
\begin{aligned}
{1\over \mu_j^2} {1\over 1+ |{x -\xi_j \over \mu_j} |^4} & {\mu_{j+1}^{n-2 \over 2} \over (\mu_{j+1}^2 + |x-\xi_{j+1} |^2)^{n-2 \over 2} }  \chi_j (x)  {\bf 1}_{\{ |x-\xi_j | < \mu_j^{1+a} \} }
\leq   \ C   {\la_{j+1}^{{n+2 \over 4}} \over (\bar \mu_{j+1})^{n-2 \over 2}} {
(\bar \mu_{j+1} )^{-2} \over \, (1+ |{x-\xi_j \over \bar \mu_{j+1}}| )^{n-2} } \chi_j (x) \\
\leq  & \ C   {\la_{j+1}^{{n-2 \over 4}} \mu_j^{2a}  \over (\bar \mu_{j+1})^{n-2 \over 2}} {
(\bar \mu_{j+1} )^{-2} \over \, (1+ |{x-\xi_j \over \bar \mu_{j+1}}| )^{n} } \chi_j (x) \leq t_0^{-\ell} \omega_{2,j+1}
\end{aligned}
$$
and also
$$
\begin{aligned}
{1\over \mu_j^2} {1\over 1+ |{x-\xi_j  \over \mu_j} |^4} & {\mu_{j+1}^{n-2 \over 2} \over (\mu_{j+1}^2 + |x-\xi_{j+1} |^2)^{n-2 \over 2} }  \chi_j (x)  {\bf 1}_{\{ |x-\xi_j | > \mu_j^{1+a} \} }
\leq   \ C  {1\over \mu_j^2} {1\over 1+ |{x-\xi_j  \over \mu_j} |^4} {\mu_{j+1}^{n-2 \over 2} \over \mu_j^{(n-2) (1+a)} } \chi_j (x) \\
\leq  & \ C   \left( {\la_{j+1} \over \la_j} \right)^{n-2 \over 2} \mu_j^{- (n-2) a} \omega_{1j} \leq t_0^{-\ell}  \omega_{1j} ,
\end{aligned}
$$
for some $\ell >0$, provided $a$ is chosen small enough.
Thus  $$\Big | f' (U_j) \left( \sum_{i \geq j} U_i \right) \chi_j \Big | \,
\, \le\, t_0^{-\ell} \left( \omega_{1j} +  \omega_{2,j+1} \right)  $$ for some $\ell >0$. Another term in $\bar E_{11}$ is $$ \sum_{j=2}^k f'(U_j ) (U_{j-1} - U_{j-1}(0) ) \chi_j, $$ which can be bounded as follows
$$
\left| f'(U_j ) (U_{j-1} - U_{j-1}(0) ) \chi_j \right| \, \le\, C\,\la_j {\la_j^{{n-2 \over 2} } \over \mu_j^{n-2 \over 2}} {\mu_j^{-2} \over (1+ |{x -\xi_j  \over \mu_j} | )^{2+a} } \chi_j \leq t_0^{-\ell} \omega_{1j}
$$
 for some  $\ell >0$.
Let us now consider the term $$\sum_{j=2}^k \big[ N_{U_j} (\sum_{l\not= j, j-1} U_l ) -\sum_{l\not= j} f( U_l ) \big] \chi_j $$
 in $\bar E_{11}$.
By construction, $$\Big|  N_{U_j} \Big(\sum_{l\not= j, j-1} U_l\Big ) \chi_j \Big | \, \le\, C\,\left( |U_{j-1} |^p + |U_{j+1} |^p \right) \chi_j. $$ We have
$$
|U_{j-1} |^p \chi_j \, \le\, C\,\la_j (t_0 )^{1-{a \over 2} }  \, {\la_j^{{n-2 \over 2} } \over \mu_j^{n+2 \over 2}} {1 \over (1+ |{x -\xi_j  \over \mu_j} | )^{2+a} }  \chi_j \leq t_0^{-\ell}  \, {\la_j^{{n-2 \over 2} } \over \mu_j^{n+2 \over 2}} {1 \over (1+ |{x -\xi_j \over \mu_j} | )^{2+a} }  \chi_j
$$
and
$$
|U_{j+1} |^p \chi_j \, \le\, C\,  \,  {\la_{j+1}^{{n+2 \over 4} } \over (\bar \mu_{j+1})^{n+2 \over 2}} {
1 \over \, (1+ |{x-\xi_j \over \bar \mu_{j+1}}| )^{n} } \chi_j \leq t_0^{-\ell} \omega_{2, j+1} .
$$
To bound $|f( U_l )| \chi_j$, for $l\not= j$, we argue in a very similar way.
We thus conclude that
$$
\Big \| \sum_{j=2}^k \big[ N_{U_j} (\sum_{l\not= j, j-1} U_l ) -\sum_{l\not= j} f( U_l ) \big] \chi_j \Big\|_{a,\sigma'',\beta''} \ \leq\ t_0^{-\ell}.
$$
The next term  to estimate in $\bar E_{11}$ is $$\bar \chi \sum_{j=2}^k (1-\chi_j ) \pp_t U_j. $$
Let us fix $j$, and observe that
$$
\left| \bar \chi  (1-\chi_j ) \pp_t U_j \right| \, \le\, C\,{\la_j^{n-2 \over 2} \over \mu_j^{n+2 \over 2}} U ({x-\xi_j \over \mu_j} ) \bar \chi (1- \chi_j ).
$$
If $j\not= 2$, then
$$
{\la_j^{n-2 \over 2} \over \mu_j^{n+2 \over 2}} U ( {x-\xi_j \over \mu_j} ) \bar \chi (1- \chi_j ) \, \le\, t_0^{-\ell} {\la_j^{{n-2 \over 2} } t^{-\sigma''} \over \mu_j^{n-2 \over 2}} {\mu_j^{-2} \over (1+ |{x-\xi_j  \over \mu_j} | )^{2+a} } \bar \chi_j.
$$
If $j=2$, then
$$
{\la_2^{n-2 \over 2} \over \mu_2^{n+2 \over 2}} U ( {x-\xi_2 \over \mu_2} ) \bar \chi (1- \chi_2  ) \, \le\, C\,\la_2^{n-6 \over 6} \,
{\la_{2}^{{n+2 \over 4} } \over (\bar \mu_2)^{n-2 \over 2}} {
(\bar \mu_2 )^{-2} \over \, (1+ |{x-\xi_2 \over \bar \mu_2}| )^n }.
$$
A similar bound is also valid for the last terms in $\bar E_{11}$.
We conclude that
$$
\| \bar E_{11} \|_{a, \sigma'' , \beta'' } \, \le\, \,t_0^{-\ell}
$$
for some positive $\ell$.

\medskip
\subsubsection*{ Estimate of $E_2$ in \eqref{e1}} \ \
We start with $\left( \bar \chi^p - \bar \chi \right) f (\sum_{j=1}^k U_j )$: we have
$$
\left| \left( \bar \chi^p - \bar \chi \right) f (\sum_{j=1}^k U_j ) \right| \, \le\, C\,{|\bar \chi^p - \bar \chi| \over (1+ |x|)^{n+2}}
\, \le\, C\,{ 1\over (t^{1\over 2} + |x| )^{n+2}}.
$$
The second term in $E_2$ can be estimated as follows
$$
 \left| ( \Delta_x - \pp_t  ) (\bar \chi ) (\sum_{j=1}^k U_j ) \right| \, \le\, {C \over t } { \chi ({|x| \over \sqrt{t}}) \over (1+ |x|)^{n-2}} \ \le\
 {C\over (t^{1\over 2} + |x| )^{n-1}}.
 $$
 The last term in $E_2$ can be treated in a similar way, thus we conclude that
 $$
| \bar E_2 (x,t) | \leq C t_0^{-\ell} \omega_3 (x,t)
 $$
for some $\ell >0$.

\medskip
\subsubsection*{Estimates of the remaining terms}
The next two terms are directly bounded
$$\left| \sum_{j=2}^k  {(-1)^{j-1} \over \mu_j^{n+ 2 \over 2}} \theta_j (\vec\mu_1 ,\vec \xi ) \chi_j \right| \, \le\, C\,{\mu_{1 , j-1} \over \mu_{0, j-1}} \,
 \sum_{j=2}^k {\la_j^{{n-2 \over 2} } \over \mu_j^{n-2 \over 2}} {\mu_j^{-2} \over (1+ |{x-\xi_j  \over \mu_j} | )^{2+a} } \bar \chi_j,
 $$
 and
 $$
 \left| {\bar \chi \over \mu_1^{n+2 \over 2} } \theta_1 (\vec\mu_1,\vec \xi)  \right| \, \le\, C\,( \mu_{11} + {\mu_{21} \over \mu_{20}} ) \,
 \sum_{j=1}^{k-1} {\la_{j+1}^{{n+2 \over 4} } \over (\bar \mu_{j+1})^{n-2 \over 2}} {
(\bar \mu_{j+1} )^{-2} \over \, (1+ |{x-\xi_j \over \bar \mu_{j+1}}| )^{n-2} }
 $$
 From \eqref{muj0}, we conclude that
 $$
 \left \| \sum_{j=2}^k  {(-1)^{j-1} \over \mu_j^{n+ 2 \over 2}} \theta_j (\vec\mu_1   ) \chi_j + {\bar \chi \over \mu_1^{n+2 \over 2} } \theta_1 (\vec\mu_1  )   \right \|_{a,\sigma'', \beta''}\ \, \le\, \,\ t_0^{-\ell}
 $$
 for some $\ell >0$.

\medskip
 Next we estimate
$\sum_{j=2}^k \left[ 2 \nabla_x \varphi_{0j} \nabla_x (\chi_j  ) + \Delta_x (\chi_j  ) \varphi_{0j} \right] $. From \eqref{defphi0j}, we easily get that
$$
| \varphi_{0j} \chi_j | \, \le\, C\,\la_j U_j \chi_j .
$$
Since
$$
| \nabla_x \chi_j | \, \le\, C\,{1\over \bar \mu_j} {\bf 1}_{\{\bar \mu_j < |x-\xi_j |<2   \bar \mu_j \}} + C{1\over \bar \mu_{j+1}} {\bf 1}_{\{\bar \mu_{j+1} <|x-\xi_j | <2 \bar \mu_{j+1} \}}
$$
where ${\bf 1}$ is defined in \eqref{funo}.
We have
$$
\begin{aligned}
\left| 2 \nabla_x \varphi_{0j} \nabla_x (\chi_j  )  \right|&\, \le\, C\,{\la_j \over \mu_j^{n \over 2}} {1\over 1+ |{x-\xi_j \over \mu_j} |^{n-1} } \left(  {1\over \bar \mu_j} {\bf 1}_{\{\bar \mu_j < |x-\xi_j |<2   \bar \mu_j \}} + {1\over \bar \mu_{j+1}} {\bf 1}_{\{\bar \mu_{j+1} <|x-\xi_j| <2 \bar \mu_{j+1} \}} \right) \\
&\, \le\, t_0^{-\ell} \, \left( \omega_{1j}    + \omega_{2j} \right) .
\end{aligned}
$$
A very similar estimate is also valid for the term $ \Delta_x (\chi_j  ) \varphi_{0j}$. We find
$$
\left \| \sum_{j=2}^k \left( 2 \nabla_x \varphi_{0j} \nabla_x (\chi_j  ) + \Delta_x (\chi_j  ) \varphi_{0j} \right) \right \|_{a,\sigma'', \beta''}\
\le\  Ct_0^{-\ell } .
$$

\medskip
Now let us consider the term $ N_{\bar U} [\varphi_0 ]$. We have
$$
\begin{aligned}
|N_{\bar U} [\varphi_0 ] | & \leq C |\varphi_0|^p \leq \sum_{j=2}^k |\varphi_{0j}|^p \chi_j \\
&\leq C \sum_{j=2}^k {1\over \mu_j^{n+2 \over 2}} |\phi_{0j} |^p \chi_j \leq C {1\over \mu_j^{n+2 \over 2}} {\la_j^{n+2 \over 2}  \over (1+ |{x-\xi_j \over \mu_j}|)^{2+a} } \chi_j.
\end{aligned}
$$

\medskip
The remaining terms in $E^{out}$ can be treated as follows:
$$
\left| \mu_{1}^{-{n+2 \over 2}} D_1 [\vec \mu_1 ]  [\bar \chi - \eta_1]  \right| \leq C |\dot \mu_{11}| {1\over (\sqrt{t} + |x| )^{n-2}},
$$
 for each $j=2, \ldots , k-1$,
$$
\begin{aligned}
&\left| \mu_j^{-{n+2 \over 2}} D_j [\vec \mu_1] [\chi_j - \eta_j] \right| \leq {C \over R^{2-a} \mu_j^{n+2 \over 2}} {\la_j^{n-2 \over 2} \over (1+ |{x-\xi_j \over \mu_j}| )^{2+a} } \chi (|{x-\xi_j \over \bar \mu_j} |) \\
&+ \la_{j+1}^{1-{a\over 2}} {C \over  \mu_{j+1}^{n+2 \over 2}} {\la_{j+1}^{n-2 \over 2} \over (1+ |{x-\xi_j \over \mu_{j+1}}| )^{2+a} } \chi (|{x-\xi_j \over \bar \mu_{j+1}} |)
\end{aligned}
$$
and for $j=k$
$$
\begin{aligned}
&\left| \mu_k^{-{n+2 \over 2}} D_k [\vec \mu_1] [\chi_k - \eta_k] \right| \leq {C \over R^{2-a} \mu_k^{n+2 \over 2}} {\la_k^{n-2 \over 2} \over (1+ |{x-\xi_k \over \mu_k}| )^{2+a} } \chi (|{x-\xi_k \over \bar \mu_k} |) .
\end{aligned}
$$

\medskip
Collecting all these estimates, we obtain the validity of \eqref{esterrore}.
\end{proof}

\medskip
\begin{lemma} \label{lemma2} Assume the parameters $\mu_j $ have the form \eqref{muj0} with $\vec \mu_1$ satisfying $\equ{cotmu}$, and the points $\vec \xi$ satisfy \eqref{cotpoint}. Assume that $a$, $\sigma''$, $\sigma'$, $\beta'$ and $\beta'' $ are defined as in \eqref{fixthem}.  Then there exists $\ell >0$ so that for all sufficiently large $t_0$ we have
\be \label{estVpsi}
\| V \, \psi   \|_{ a , \sigma'' , \beta'' } \, \le\, t_0^{-\ell} \, \| \psi \|_{*,a, \sigma' , \beta' }
\ee
where $V $ is defined in \eqref{poto}.
\end{lemma}

\begin{proof}
Using the convention that  $\bar \mu_{k+1} =0$ and $\bar \mu_1= \sqrt{t}$, we have
$$
\begin{aligned}
|V \psi | &\leq C (1-\sum_{j=1}^k \zeta_j ) \left( \sum_{i=1}^k U_i {\bf 1}_{ \{\bar \mu_{i+1} < |x-\xi_i| < \bar \mu_i \} } \right)^{p-1}
\sum_{j=2}^k\big (\omega_{1j}^* + \omega_{2j}^* + \omega_{11}^*  + \omega_3^* \big) \, \| \psi \|_{*, a, \sigma' , \beta' } \\
& \leq C   \, \| \psi \|_{*, a', \sigma' , \beta'} \, (1-\sum_{j=1}^k \zeta_j ) \,  \sum_{i=1}^k   U_i^{p-1}  {\bf 1}_{ \{\bar \mu_{i+1} < |x-\xi_i| < \bar \mu_i \} }
\sum_{j=2}^k\big (\omega_{1j}^* + \omega_{2j}^* + \omega_{11}^*  + \omega_3^* \big)
\end{aligned}
$$
where $\omega_{1j}^* $, $ \omega_{2j}^* $,  $\omega_{11}^*  $, and $ \omega_3^*$ are defined respectively in \eqref{omega1j}, \eqref{omega2j},
\eqref{omega1} and \eqref{omega3}. Also the norm $\| \cdot \|_{*,a' ,\sigma' , \beta'}$ is defined in \eqref{norm1*}.

Our purpose is now to show that, for all $i \in \{ 1, \ldots , k \}$,
\be \label{ww0}
\|  (1-\sum_{j=1}^k \zeta_j ) \,    U_i^{p-1}  {\bf 1}_{ \{\bar \mu_{i+1} < |x-\xi_i | < \bar \mu_i \} }
\sum_{j=2}^k\big (\omega_{1j}^* + \omega_{2j}^* + \omega_{11}^*  + \omega_3^* \big) \|_{ a , \sigma'' , \beta'' } \, \le\, \,t_0^{-\ell}
\ee
Indeed, estimate \eqref{estVpsi}  follows directly  from \eqref{ww0}.

Take $i =k$, then
\be \label{ww1}
\begin{aligned}
(1-\sum_{j=1}^k \zeta_j ) & U_k^{p-1}  {\bf 1}_{ \{ |x-\xi_k | < \bar \mu_k \} }
\sum_{j=2}^k\big (\omega_{1j}^* + \omega_{2j}^* + \omega_{11}^*  + \omega_3^* \big)  \\
& \leq C (1-\sum_{j=1}^k \zeta_j ) U_k^{p-1}  {\bf 1}_{ \{ |x-\xi_k | < \bar \mu_k \} }
\big (\omega_{1,k}^* + \omega_{2k}^*  \big) \\
&\leq C {1\over R^{2-a}} {t^{-\sigma''} \over \mu_k^{n+2 \over 2} } {\la_k^{n-2 \over 2} \over (1+ |{x-\xi_k \over \mu_k} |)^{2+a} } \chi ({|x-\xi_k | \over \bar \mu_k}) \leq   {C\over R^{2-a}} \omega_{1,k} .
\end{aligned}
\ee
Take $i \in \{ 2, \ldots , k-1 \}$, then
\be \label{ww2}
\begin{aligned}
& (1-\sum_{j=1}^k \zeta_j )  U_i^{p-1}  {\bf 1}_{ \{ \bar \mu_{i+1} <|x-\xi_i | < \bar \mu_i \} }
\sum_{j=2}^k\big (\omega_{1j}^* + \omega_{2j}^* + \omega_{11}^*  + \omega_3^* \big)  \\
& \leq C (1-\sum_{j=1}^k \zeta_j ) U_i^{p-1}  {\bf 1}_{ \{ \bar \mu_{i+1} <|x -\xi_i | < \bar \mu_i \} }
\big (\omega_{1,i+1}^* + \omega_{2, i+1}^*   \big ) \\
&+  C (1-\sum_{j=1}^k \zeta_j ) U_i^{p-1}  {\bf 1}_{ \{ \bar \mu_{i+1} <|x-\xi_i  | < \bar \mu_i \} }
\big (\omega_{1,i}^* + \omega_{2, i}^*   \big ) \\
& + C (1-\sum_{j=1}^k \zeta_j ) U_i^{p-1}  {\bf 1}_{ \{ \bar \mu_{i+1} <|x-\xi_i  | < \bar \mu_i \} } \big ( \omega_{1,i-1}^* + \omega_{2, i-1}^* + \omega_{11}^* + \omega_3^* \big) .
\end{aligned}
\ee
We next estimate the different terms in the above expression. Let us consider first the terms involving $(\omega_{1,i+1}^* + \omega_{2, i+1}^* )$. We have
$$
\begin{aligned}
(1-\sum_{j=1}^k \zeta_j ) & U_i^{p-1}  {\bf 1}_{ \{ \bar \mu_{i+1} <|x -\xi_i | < \bar \mu_i \} } \omega_{1,i+1}^* \leq {C \over R^4}   {\bf 1}_{ \{ \bar \mu_{i+1} <|x-\xi_i | < \bar \mu_i \} } {t^{-\sigma''} \over \mu_{i+1}^{n-2 \over 2}}  \, { {\la_{i+1}^{{n-2 \over 2} } \over (1+ |{x -\xi_i \over \mu_{i+1}} | )^{a } }}\,  \chi
\left( { |x-\xi_i | \over \bar \mu_{i+1}} \right) \\
&\leq {C \over R^4} \la_{i+1} \omega_{1, i+1}
\end{aligned}
$$
and
$$
\begin{aligned}
(1-\sum_{j=1}^k \zeta_j ) & U_i^{p-1}  {\bf 1}_{ \{ \bar \mu_{i+1} <|x-\xi_i | < \bar \mu_i \} } \omega_{2,i+1}^* \leq
(1-\sum_{j=1}^k \zeta_j ) & U_i^{p-1}  {\bf 1}_{ \{ \bar \mu_{i+1} <|x-\xi_i | <  \mu_i \} } \omega_{2,i+1}^*  \\
& + (1-\sum_{j=1}^k \zeta_j )  U_i^{p-1}  {\bf 1}_{ \{ \mu_{i} <|x-\xi_i | < \bar \mu_i \} } \omega_{2,i+1}^*  \\
& \leq {C \over R^4} \omega_{2, i+1} + {t^{-b} \over R^{2-a}} \omega_{1,i},
\end{aligned}
$$
for some $b>0$.

Let us now consider  the terms involving $(\omega_{1,i}^* + \omega_{2, i}^* )$. A direct computation gives
$$
(1-\sum_{j=1}^k \zeta_j )  U_i^{p-1}  {\bf 1}_{ \{ \bar \mu_{i+1} <|x-\xi_i | < \bar \mu_i \} } \omega_{1,i}^* \leq
{C\over R^2} \omega_{1,i}
$$
and
$$
(1-\sum_{j=1}^k \zeta_j )  U_i^{p-1}  {\bf 1}_{ \{ \bar \mu_{i+1} <|x-\xi_i | < \bar \mu_i \} } \omega_{1,i}^* \leq
{C\over R^2} \omega_{1,i}.
$$
Arguing similarly, we get
$$
(1-\sum_{j=1}^k \zeta_j ) U_i^{p-1}  {\bf 1}_{ \{ \bar \mu_{i+1} <|x-\xi_i | < \bar \mu_i \} } \big ( \omega_{1,i-1}^* + \omega_{2, i-1}^* + \omega_{11}^* + \omega_3^* \big) \leq {C \over R^{2-a}} \omega_{1,i}.
$$
 Thus we conclude that, for $i \in \{ 2, \ldots , k-1 \}$ it holds
$$
 (1-\sum_{j=1}^k \zeta_j )  U_i^{p-1}  {\bf 1}_{ \{ \bar \mu_{i+1} <|x -\xi_i | < \bar \mu_i \} }
\sum_{j=2}^k\big (\omega_{1j}^* + \omega_{2j}^* + \omega_{11}^*  + \omega_3^* \big) \leq {1\over R^{2-a}} \left( \omega_{1, i} + \omega_{1, i+1} + \omega_{2, i+1} \right).
$$
Take now $i=1$,
then
\be \label{ww3}
\begin{aligned}
(1-\sum_{j=1}^k \zeta_j ) & U_1^{p-1}  {\bf 1}_{ \{ \bar \mu_2 <|x -\xi_1 | < \sqrt{t}  \} }
\sum_{j=2}^k\big (\omega_{1j}^* + \omega_{2j}^* + \omega_{11}^*  + \omega_3^* \big)  \\
& \leq C (1-\sum_{j=1}^k \zeta_j ) U_1^{p-1}  {\bf 1}_{ \{ \bar \mu_2 <|x -\xi_1 | < \sqrt{t}  \} }
\big (\omega_{1,2}^* + \omega_{22}^*  + \omega_{11}^* + \omega_3^*\big)  .
\end{aligned}
\ee
We have
$$
\begin{aligned}
(1-\sum_{j=1}^k \zeta_j ) &U_1^{p-1}  {\bf 1}_{ \{ \bar \mu_2 <|x -\xi_1 | < \sqrt{t}  \} }
\big (\omega_{1,2}^* + \omega_{22}^* \big ) \\
&\leq t^{-\sigma'' + \sigma'} \bar \mu_2^2 \omega_{1,2} + t^{-\sigma''+ \sigma' }  {\mu_2^2 \over \bar \mu_2^2} \bar \mu_2^{\sigma'' -\sigma' } \omega_{22} \\
&\leq  t_0^{-\ell}  (\omega_{1,2} + \omega_{2,2} )
\end{aligned}
$$
for some $\ell >0$, thanks to the facts that $\sigma'' > \sigma'$ and that $\sigma''$ and $\sigma'$ are small positive numbers. Also,
$$
\begin{aligned}
(1-\sum_{j=1}^k \zeta_j ) & U_1^{p-1}  {\bf 1}_{ \{ \bar \mu_2 <|x-\xi_1 | < \sqrt{t}  \} }
\omega_{11}^*\leq {t^{-\sigma' + \sigma} \over (1+ |x-\xi_1|)^{2-\sigma'' + \sigma'}} \omega_{1,1} \leq  t_0^{-\ell} \omega_{11},
\end{aligned}
$$
and
$$
\begin{aligned}
(1-\sum_{j=1}^k \zeta_j ) & U_1^{p-1}  {\bf 1}_{ \{ \bar \mu_2 <|x-\xi_1 | < \sqrt{t}  \} }
\omega_{3}^*\leq {C \over 1+ |x|^4} {1 \over (\sqrt{t} + |x| )^{\beta' -2}  }\\
&+ C {t^{- {\beta' -2 \over 2}  } \over 1+ |x|^4} \chi ({|x| \over \sqrt t} )
\leq t_0^{-\ell} \left( \omega_{11} + \omega_3 \right)
\end{aligned}
$$
thanks to the fact that $\sigma''$ can be chosen so that  ${\beta' -2 \over 2} -1 -\sigma'' > 0$ since $\beta' >2+ \alpha$.

The validity of \eqref{estVpsi} thus follows from \eqref{ww0}, \eqref{ww1}, \eqref{ww2} and \eqref{ww3}.
\end{proof}

Let us now estimate the linear operator $B[\vec\phi]$. We recall that, using summation convention,
$$ \begin{aligned}
B[\vec \phi]  \ =&\      \frac{2} { \mu_j^{\frac{n-2}2}} \nn_x\eta_j\cdot \nn_x \phi_j  +  \frac{1} { \mu_j^{\frac{n-2}2} }(-\eta_{jt} +  \Delta_x \eta_j )  \phi_j \\   + &  \frac{\mu_j \dot \mu_j} { \mu_j^{\frac{n+2}2}} [ \phi_j  + y \cdot \nn_y \phi_j ]\,\eta_j  + \frac{\mu_j \dot \xi_j} { \mu_j^{\frac{n+2}2}} \cdot \nn_y \phi_j \,\eta_j   \\
+& \eta_j (f'(u_*) - f'(U_j))
\frac{\phi_j} { \mu_j^{\frac{n-2}2} }
\\ = &\ B_1 + B_2+ B_3 + B_4
\end{aligned} $$

\begin{lemma}
 Assume the parameters $\mu_j $ have the form \eqref{muj0} with $\vec \mu_1$ satisfying $\equ{cotmu}$, and the points $\vec \xi$ satisfy \eqref{cotpoint}. Assume that $a$, $\sigma''$ and $\beta'' $ are defined as in \eqref{fixthem}.  Then there exists $\ell >0$ so that for all sufficiently large $t_0$ we have
$$
\| B[\vec\phi]\, \|_ {a ,\sigma'', \beta''} \ \le \  t_0^{-\ell} \|\vec \phi\|_{ \sigma} .
$$

\end{lemma}

\proof

We recall that by definition of $\| \phi_j\|_{j, \sigma} $ we have
$$
(1+ |y|) |\nn_y\phi(y,t)|  +  |\phi(y,t)| \ \le\  \la_j^{\frac{n-2}2} t^{-\sigma}  \frac{  R^{n}} {1+|y|^{n+1}} \| \phi_j\|_{j, \sigma}
$$
Hence, by choosing $\ve$ sufficiently small in $R= t^\ve$ we can assume $\la_j^{\frac{n-2}2} R^{n} \le t^{-2\sigma} $,
$$
|B_3(x,t)| \ \le \ \frac{\la_j^{\frac{n-2}2}t^{-3\sigma} }{ \mu_j^{\frac{n+2}2}} \frac{\|\phi_j\|_{j,\sigma } }{1+|\mu_j^{-1} (x-\xi_j) |^{n+1}}
{\bf 1} _{\{ |x-\xi_j | < 2R\mu_j \}}
$$
As a conclusion, the following bound holds: for a $\ell>0$ that can be chosen arbitrarily small, we have
$$
\| B_3\|_ { a,2\sigma, \beta\\} \ \le \  t_0^{-\ell} \sum_{j=1}^k\|\phi_j\|_{j, \sigma} .
$$
Similarly we estimate the term $B_2$:
$$\begin{aligned}
|B_2(x,t)| \ \le &  \
\frac{\la_j^{\frac{n-2}2}t^{-\sigma} }{ \mu_j^{\frac{n+2}2}} \frac{\|\phi_j\|_{j,\sigma  } }{R^{3} }
 {\bf 1}_{\{ |x-\xi_j | < 2R\mu_j \}}   \\  \le\ & \ \frac 1{R^{1-a}}
\frac{\la_j^{\frac{n-2}2}t^{-\sigma} }{ \mu_j^{\frac{n+2}2}}
\frac  {\|\phi_j\|_{j,\sigma, a  }}{ 1 + |\mu_j^{-1} (x-\xi_j)|^{2+ a}} .
\end{aligned} $$
The same estimate holds for $B_1$. We find that for some $\sigma''>\sigma$,
$$
\| B_1\|_ {a ,\sigma'', \beta''} + \| B_3\|_ { a ,\sigma'', \beta''} \ \le \  \delta \sum_{j=1}^k\|\phi_j\|_{j, \sigma} .
$$
Finally, we estimate $B_4$. We have that
$$
\begin{aligned}
|\eta_j (f'(u_*) - f'(U_j))|{\bf 1 }_{\{ |x| < \bar \mu_{j+1}\} } \ \le&\ f'\big( \sum_{i=j+1}^k U_i {\bf 1 }_{\{\bar \mu_{i+1}<  |x-\xi_i | < \bar \mu_{i}\} } \big )\\
\le&\    C \sum_{i=j+1}^k  U_i^{p-1} {\bf 1 }_{\{\bar \mu_{i+1}<  |x-\xi_i | < \bar \mu_{i}\} }\\\le  &C \sum_{i=j+1}^k \frac 1{\mu_i^2}
 \frac 1{ 1 + |\mu_i^{-1} (x-\xi_i) |^4}\,  {\bf 1 }_{\{ |x-\xi_i | < \bar \mu_{i}\} }.
\end{aligned} $$
Hence
$$
\begin{aligned}
&|\eta_j (f'(u_*) - f'(U_j))| \frac {|\phi|} {\mu_j^{\frac{n-2}2} }\,{\bf 1 }_{\{ |x-\xi_i | < \bar \mu_{j+1}\} } \\
\ \le &\ C \sum_{i=j+1}^k   \frac 1{\mu_i^{\frac{n+2}2}}
 \frac 1{ 1 + |\mu_i^{-1} (x-\xi_i ) |^4}\,  {\bf 1 }_{\{ |x-\xi_i | < \bar \mu_{i}\} } \left(\frac{ \mu_i} {\mu_j}\right)^{\frac{n-2}2}
 \la_j^{\frac{n-2}2} t^{-\sigma}  R^{n} \,\|\vec \phi\|_{\sigma}
\end{aligned} $$
Using that
$\left(\frac{ \mu_i} {\mu_j}\right)^{\frac{n-2}2} \le \la_i^{\frac{n-2}2} $ and that with a convenient choice of $\ve$ we have that
$\la_j^{\frac{n-2}2}   R^{n+1-a} \le t^{-2\sigma}$, we find
\be\label{nana}
\begin{aligned}
&\ |\eta_j (f'(u_*) - f'(U_j))| \frac {|\phi|} {\mu_j^{\frac{n-2}2} }\,{\bf 1 }_{\{ |x-\xi_j | < \bar \mu_{j+1}\} } \\
\ \le &\ C \sum_{i=j+1}^k   \frac 1{\mu_i^{\frac{n+2}2}}
 \frac 1{ 1 + |\mu_i^{-1} (x-\xi_i)|^4}\,  {\bf 1 }_{\{ |x-\xi_i | < \bar \mu_{i}\} } \la_i^{\frac{n-2}2}
 t^{-3\sigma}  \,\|\vec \phi\|_{\sigma}.
\end{aligned} \ee

On the other hand, we can estimate
$$\begin{aligned}
|\eta_j (f'(u_*) - f'(U_j))|{\bf 1 }_{\{|x-\xi_j | >\bar \mu_{j+1}\}} \le & \ CU_j^{p-1}  (\frac  {U_{j+1}}{U_j}   +  \frac  {U_{j-1}}{U_j} )
\end{aligned} $$
We have that in the region $\bar \mu_{j+1}< |x-\xi_j | < \bar \mu_{j} $ ( equivalently   $\la_{j+1}^{\frac 12} <|y|< \la_{j}^{-\frac 12} $
for $y= \frac {x-\xi_j}{\mu_j}$).
$$\begin{aligned}
\frac  {U_{j+1}}{U_j}\ \le&\    \la_{j+1}^{\frac{n-2}2}
 \frac {1+|y|^{n-2}} { \la_{j+1}^{n-2}  + |y|^{n-2}} \\ \le\  &\
 \frac{ \la_{j+1}^{\frac{n-2}2}} { \la_{j+1}^{\frac{n-2}2}  + |y|^{n-2}}   +
\la_{j+1}^{\frac{n-2}2}.\\
\frac  {U_{j-1}}{U_j}\ \le &\  \la_j^{\frac{n-2}2} \frac{1+  | y|^{n-2}}{1+  |\la_j y|^{n-2}}\\
 \le& \
 \la_j^{\frac{n-2}2} ({1+  | y|^{n-2}})
\end{aligned} $$
Hence
$$\begin{aligned}
\eta_j U_j^{p-1}  \frac  {U_{j-1}}{U_j}    \le & \
\frac C{\mu_j^2} \frac  {\la_j^{\frac{n-2}2}} {1+|y|^{2+a}}  (1+  |y|^{n-4+a} )   \\
\le & \   \la_j^{1 -\frac a2}  \frac C{\mu_j^2}\frac 1  {1+|y|^{2+a}} {\bf 1 }_{\{|x-\xi_j| <\bar \mu_{j}\}},\\
\eta_j U_j^{p-1}  \frac  {U_{j+1}}{U_j}\le & \    \la_{j+1}^{\frac {n-2}2}
\frac C{\mu_j^2}\frac 1  {1+|y|^{4}} {\bf 1 }_{\{|x-\xi_j | <\bar \mu_{j}\}} \\ & +\
      \frac C{\bar\mu_{j+1}^2} \frac 1{1+|y|^4}  \frac{ \la_{j+1}^{\frac{n}2}} { \la_{j+1}^{\frac{n-2}2}  + |y|^{n-2}}\\
     \le &  \quad  \frac C{\bar\mu_{j+1}^2}  \frac{ 1} { 1  + |\bar\mu_{j+1}^{-1}(x-\xi_j) |^{n}}
\end{aligned} $$
Then, assuming with no loss of generality that $\ve$ in $R= t^\ve$ is chosen so that  $$\la_{j+1}^{\frac {n-2}4} R^{n}\le t^{-2\sigma}, \quad \la_j^{1 -\frac a2} R^{n}\le t^{-2\sigma}, $$ we find
\be\label{nana1}
\begin{aligned}
\ |\eta_j (f'(u_*) - f'(U_j))| \frac {|\phi|} {\mu_j^{\frac{n-2}2} }\,{\bf 1 }_{\{ |x-\xi_j | > \bar \mu_{j+1}\} }
\ \le  \
\frac C{\bar\mu_{j+1}^{\frac{n+2}2}}  \frac{ 1} { 1  + |{\bar\mu}_{j+1}^{-1} (x-\xi_j) |^{n}} \la_{j+1}^{\frac {n-2}4} t^{-3\sigma} \|\phi\|_{a,\sigma}\\  +\
 \frac C{\mu_j^2}\frac 1  {1+|\mu_j^{-1}(x-\xi_j) |^{2+a}} {\bf 1 }_{\{|x-\xi_j | <\bar \mu_{j}\}} \la_{j}^{\frac {n-2}2} t^{-3\sigma}\|\phi_j\|_{j, a,\sigma}.
\end{aligned} \ee
Combining estimates \equ{nana} and \equ{nana1} we find
$$
\| B_4\|_ {a , 2\sigma, \beta''} \ \le \  \delta \sum_{j=1}^k\|\phi_j\|_{j,a, \sigma} .
$$
At last we get for the full operator and for number  $\sigma'' >\sigma$,
$$
\| B[\vec \phi]\|_{a, \sigma'',\beta''}  \ \le\  \delta \sum_{j=1}^k\|\phi_j\|_{j,\sigma} .
$$
\qed

\bigskip
Finally, let us consider the remaining terms nonlinear term $VZ_* +  \mathcal N(\vec\phi, \psi, \vec \mu_1)$.

\begin{lemma}
We have the validity of the estimate
$$
\| VZ_* +  \mathcal N(\vec\phi, \psi, \vec \mu_1)\|_{a,\sigma'',\beta''}   \ \le \ \delta\, [ 1 +  \|\vec \phi\|^p_\sigma +
\|\psi \|_{*,a,\sigma',\beta'}^p]
$$
\end{lemma}

\proof
We see that
$$\begin{aligned}
 \big| N_{u^*}\Big ( \sum_{j=1}^k \mu_j^{-\frac{n-2}2} \phi_j\eta_j +\psi + Z^*\Big )  \big | \le
  \frac 1{\mu_j^{\frac{n+2}2} }|\phi_j|^p\eta_j + |\psi|^p + |Z^*|^p =:  N_1 + N_2 + N_3
 \end{aligned} $$
 We have that
 $$
 \frac 1{\mu_j^{\frac{n+2}2} }|\phi_j|^p\eta_j  \le    \frac 1{\mu_j^{\frac{n+2}2} } \|\phi_j\|_{j,\sigma} ^p
 \frac 1{1+ |y|^{4}}  \la_j^{\frac{ n-2}2}t^{-\sigma}  t^{-(p-1)\sigma} R^{pn}\la_j^2
 $$
 Assuming that $R^{pn}\la_j^2 \le t^{ (p-3)\sigma} $ we then find
 $$
 \| N_1\|_ {a , 2\sigma, \beta''}  \ \le\  \delta \|\phi_j\|_{j, \sigma}.
 $$
 Similarly,

 $$
 |N_2|\  \le \  \|\psi\|^p_{*,a, \sigma',\beta'} \, \sum_{j=2}^k\big ({\omega_{1j}^*}^p +  {\omega_{2j}^*}^p+  {\omega_{11}^*}^p  + {\omega_{3}^*}^p \big)
 $$
 Since $\beta'-2 > \frac{n-2}2 $ we have $ (\beta'-2 )p >  \beta'$.
 We may assume $p(n-\sigma-2) >  n-\sigma'' $,
 hence for some  $\gamma >0$, we may assume
 $$
{\omega_{2j}^*}^p+  {\omega_{11}^*}^p  + {\omega_{3}^*}^p   \ \le\  C t^{-\gamma}\,\big( \omega_{2j}  +   \omega_{11}   +\omega_{3}^{1+\gamma}\big)
 $$
Finally we see that for $j\ge 2$
$$
{w_{1j }^*}^p  \ \le  \  w_{1j} ( 1 + |\mu_j^{-1} (x-\xi_j)|^{2- (p-1) a}   )\la_j^2 \ \le\ w_{1j} \la_j
$$
And, as a conclusion for some numbers $\sigma'' >\sigma'$,  $\beta''>\beta'$  we  get
$$ \| N_2\|_{a, \sigma'',\beta''}  \le   \delta \| \psi \|_{*,a, \sigma',\beta'}^p   $$
Finally, using estimate \equ{estZ*} on $Z_*$ and $\beta > \frac{n-2}2$ we readily see that
$$  \| VZ^* \|_{a, \sigma'',\beta''}  + \| N_3\|_{a, \sigma'',\beta''}  \le   \delta    $$
for an arbitrarily small $\delta$.\qed

\bigskip

\bigskip\noindent
{\bf Acknowledgements:}
 M.~del Pino has been supported by a UK Royal Society Research Professorship and Grant PAI AFB-170001, Chile. M. Musso has been  partly supported by
 Fondecyt grant 1160135, Chile.  The  research  of J.~Wei is partially supported by NSERC of Canada.

\end{document}